\documentclass{amsproc}
\usepackage{amsmath,amsthm,amsfonts,amssymb}
\usepackage{color}
\newtheorem{theorem}{Theorem}[section]
\newtheorem{lemma}[theorem]{Lemma}
\newtheorem{proposition}{Proposition}[section]
\theoremstyle{definition}
\newtheorem{definition}[theorem]{Definition}

\theoremstyle{remark}
\newtheorem{remark}[theorem]{Remark}

\numberwithin{equation}{section}

\newcommand{\R}{{\mathbb R}}

\newcommand{\N}{{\mathbb N}}

\newcommand{\essinf}{\mathop{\rm {ess\,inf}}\limits}   

\newcommand{\eqdef}{\stackrel{{\rm {def}}}{=}}   
\newcommand{\wstarconverge}{\stackrel{*}{\rightharpoonup}}

\newcommand{\dist}{\mathop{\mathrm{dist}}}

\begin{document}

\title{A singular parabolic equation: existence, stabilization}

\author{Mehdi Badra}
\address{LMAP (UMR 5142), Bat.~IPRA, Universit\'e de Pau et des Pays de l'Adour, Avenue de l'Universit\'e, 64013 cedex Pau, France}
\email{mehdi.badra@univ-pau.fr}
\author{Kaushik Bal}
\address{LMAP (UMR 5142), Bat.~IPRA, Universit\'e de Pau et des Pays de l'Adour, Avenue de l'Universit\'e, 64013 cedex Pau, France}
\email{kaushik.bal@univ-pau.fr}

\author{Jacques Giacomoni}
\address{LMAP (UMR 5142), Bat.~IPRA, Universit\'e de Pau et des Pays de l'Adour, Avenue de l'Universit\'e, 64013 cedex Pau, France}
\email{jgiacomo@univ-pau.fr}


\begin{abstract}
We investigate the following quasilinear parabolic and singular equation,
\begin{equation}
\nonumber
\tag{{\rm P$_t$}}
\left\{
\begin{aligned}
& u_t-\Delta_p u =\frac{1}{u^\delta}+f(x,u)\;\text{ in }\,(0,T)\times\Omega,\\
& u =0\,\text{ on} \;(0,T)\times\partial\Omega,\quad u>0 \text{ in }\, (0,T)\times\Omega,\\
&u(0,x) =u_0(x)\;\text{ in }\Omega,
\end{aligned}
\right.
\end{equation}
where $\Omega$ is an open bounded domain with smooth boundary in $\R^{\rm N}$, 
$1 < p< \infty$, 
$0<\delta$ and $T>0$. We assume that $(x,s)\in\Omega\times\R^+\to f(x,s)$ is a bounded below Caratheodory function, locally Lipschitz with respect to $s$ uniformly in $x\in\Omega$ and asymptotically sub\--homogeneous, i.e.
\begin{equation}
\label{sublineargrowth}
0 \leq\displaystyle\lim_{t\to +\infty}\frac{f(x,t)}{t^{p-1}}=\alpha_f<\lambda_1(\Omega),
\end{equation}
(where $\lambda_1(\Omega)$ is the first eigenvalue of $-\Delta_p$ in $\Omega$ with homogeneous Dirichlet boundary conditions) and $u_0\in L^\infty(\Omega)\cap W^{1,p}_0(\Omega)$, satisfying a cone condition defined below.
Then, for any $\delta\in (0,2+\frac{1}{p-1})$, we prove the existence and the uniqueness of a weak solution $u \in{\bf V}(Q_T)$ to $({\rm P_t})$. Furthermore, $u\in C([0,T], W^{1,p}_0(\Omega))$ and the restriction $\delta<2+\frac{1}{p-1}$ is sharp. The proof involves a semi-discretization in time approach and the study of the stationary problem associated to $({\rm P_t})$.
The key points in the proof is to show that $u$ belongs to the cone ${\mathcal C}$ defined below and by the weak comparison principle that $\frac{1}{u^\delta}\in L^\infty(0,T;W^{-1,p'}(\Omega))$ and $u^{1-\delta}\in L^\infty(0,T;L^1(\Omega))$. When $t\to \frac{f(x,t)}{t^{p-1}}$ is non\--increasing for a.e. $x\in\Omega$, we show that $u(t)\to u_\infty$ in $L^\infty(\Omega)$ as $t\to\infty$, where $u_\infty$ is the unique solution to the stationary problem. This stabilization property is proved by using the accretivity of a suitable operator in $L^\infty(\Omega)$.
 
Finally, in the last section we analyse the case $p=2$. Using the interpolation spaces theory and the semigroup theory, we prove the existence and the uniqueness of weak solutions to $({\rm P}_t)$ for any $\delta>0$ in $C([0,T], L^2(\Omega))\cap L^\infty(Q_T)$ and under suitable assumptions on the initial data we give additional regularity results. Finally, we describe their asymptotic behaviour in $L^\infty(\Omega)\cap H^1_0(\Omega)$ when $\delta<3$. 
\end{abstract}

\maketitle


\specialsection*{Introduction}
In the present paper we investigate the following quasilinear and singular parabolic problem~:
\begin{equation}
\nonumber
\tag{{\rm P$_t$}}
\left\{
\begin{aligned}
& u_t-\Delta_p u =\frac{1}{u^\delta}+f(x,u)\;\text{ in }\,Q_T,\\
& u =0\,\text{ on} \;\Sigma_T,\quad u>0 \text{ in }\, Q_T,\\
&u(0,x) =u_0(x)\;\text{ in }\Omega,
\end{aligned}
\right.
\end{equation}
where $\Omega$ is an open bounded domain with smooth boundary in $\R^{\rm N}$ (with ${\rm N}\geq 2$), $1 < p< \infty$, $0<\delta$, $T>0$, $Q_T=(0,T)\times\Omega$ and $\Sigma_T=(0,T)\times\partial\Omega$. 
We assume that $f$ is a bounded below Caratheodory function, locally Lipschitz with respect to the second variable uniformly in $x\in\Omega$ and satisfying \eqref{sublineargrowth} and $u_0\in L^\infty(\Omega)\cap W^{1,p}_0(\Omega)$. 
Such a problem arises in different models: non newtonian flows, chemical hetereogeneous catalyst kinetics, combustion. We refer to the survey {\sc Hern\'andez-Mancebo-Vega} \cite{HeMaVe},  the book {\sc Ghergu-Radulescu} \cite{GhRa} and the bibliography therein for more details about the corresponding models. 
One of our main goals is to prove the existence and the uniqueness of the weak solution to $({\rm P}_t)$. We define the notion of weak solution for the following more general problem
\begin{equation}
\nonumber
({\rm S}_t)\left\{
\begin{aligned}
&u_t-\Delta_p u =\frac{1}{u^\delta}+h(x,t) \;\text{ in }\,Q_T,\\
&u =0\,\text{ on} \;\Sigma_T,\quad u>0 \text{ in }\, Q_T,\\
& u(0,x) =u_0(x)\;\text{ in }\Omega,
\end{aligned}
\right.
\end{equation} 
where $0<T$, $h\in L^\infty(Q_T)$, $0<\delta<2+\frac{1}{p-1}$, $u_0\in L^\infty(\Omega)\cap W^{1,p}_0(\Omega)$, as follows
\begin{definition}
\begin{eqnarray*}
\nonumber
{\bf V}(Q_T)\eqdef\left \{u\,:\,u\in L^\infty(Q_T),\, u_t \in L^2(Q_T),\, u\in L^\infty(0,T;W^{1,p}_0(\Omega))\right \}
\end{eqnarray*}
\end{definition}
and 
\begin{definition}\label{Def0}
 A {\it weak} solutions to $({\rm S}_t)$ is a function $u\in {\bf V}(Q_T)$ satisfying
\begin{itemize}
\item[1.] for any compact $K\subset Q_T$, $\displaystyle\essinf_K u>0$,
\item[2.] for every test function $\phi\in{\bf V}(Q_T)$,
\begin{eqnarray*}
\int_{Q_T}\left(\phi\frac{\partial u}{\partial t}-|\nabla u|^{p-2}\nabla u\nabla\phi-\phi(\frac{1}{u^\delta}+h(t,x))\right){\rm d}x{\rm d}t=0,
\end{eqnarray*} 
\item[3.] $u(0,x)=u_0(x)$ a.e in $\Omega$.
\end{itemize}
\end{definition}
\begin{remark}
Since every $u\in {\bf V}(Q_T)$ belongs to $C(0,T;\,L^2(\Omega))$, the third point of the above definition is meaningful.
\end{remark}

 The approach we use is to study first the existence of solutions to the stationary problem $({\rm P})$ that is for $g\in L^\infty(\Omega)$, $\lambda >0$
\begin{eqnarray*}
({\rm P})\left\{\begin{array}{rll}
\displaystyle u-\lambda\big{(}\Delta_p u+\frac{1}{u^\delta}\big{)}&\displaystyle =g \;&\displaystyle \mbox{in }\,\Omega,\\
\displaystyle u&\displaystyle =0\; &\displaystyle \mbox{on }\;\partial\Omega.
\end{array}\right.
\end{eqnarray*}
To control the singular term $\frac{1}{u^\delta}$, we need to consider solutions in a conical shell ${\mathcal C}$ defined as the set of functions $v\in L^\infty(\Omega)$ such that there exists $c_1>0$ and $c_2>0$ satisfying
\begin{eqnarray*}
\nonumber
\left\{ 
 \begin{array}{ll}
 \displaystyle c_1d(x)\leq v \leq c_2d(x)\; & \displaystyle  \mbox{ if }\;\delta<1,\\
  \displaystyle  c_1d(x){\log}^{\frac{1}{p}}(\frac{k}{d(x)})\leq v \leq c_2d(x){\log}^{\frac{1}{p}}(\frac{k}{d(x)})\; & \displaystyle  \mbox{ if }\;\delta=1,\nonumber\\
  \displaystyle  c_1d(x)^{\frac{p}{\delta+p-1}}\leq v \leq c_2\left(d(x)^{\frac{p}{\delta+p-1}}+d(x)\right)\; & \displaystyle  \mbox{ if }\;\delta>1,
 \end{array}
\right.
 \end{eqnarray*}
where $d(x)\eqdef\dist(x,\partial\Omega)$ and $k>0$ is large enough. Regarding Problem $({\rm P})$, we prove the following\\
\begin{theorem}\label{1} Let $g\in L^\infty(\Omega)$ and $0<\delta< 2+\frac{1}{p-1}$. Then for any $\lambda>0$, problem $({\rm P})$ admits a unique solution $u_\lambda$ in $W^{1,p}_0(\Omega)\cap{\mathcal C}\cap C_0(\overline{\Omega})$.
%
%
\end{theorem}
Concerning the case where $\delta\geq 2+\frac{1}{p-1}$, we prove the following
\begin{theorem}\label{2} Let $g\in L^\infty(\Omega)$ and $\delta\geq 2+\frac{1}{p-1}$. Then for any $\lambda >0$, problem $({\rm P})$ admits a solution $u_\lambda$ in $W^{1,p}_{\rm loc}(\Omega)\cap{\mathcal C}\cap C_0(\overline{\Omega})$ such that $u_\lambda \not\in W^{1,p}_0(\Omega)$.
\end{theorem}
In view of establishing Theorem~\ref{3ter} below, we need to prove the following result:
\begin{theorem}\label{2bis} Let $0<\delta< 2+\frac{1}{p-1}$ and $f:\,\Omega\times\R^+\to {\mathbb R}$ be a bounded below Caratheodory function, locally Lipschitz with respect to the second variable uniformly in $x\in\Omega$, satisfying \eqref{sublineargrowth} and such that $\frac{f(x,s)}{s^{p-1}}$ is a decreasing function in $\R^+$ for a.e. $x\in\Omega$. Then there exists  a unique $u_\infty$ in $W^{1,p}_0(\Omega)\cap{\mathcal C}\cap C_0(\overline{\Omega})$ satisfying
\begin{eqnarray*}
\nonumber
({\rm Q})\left\{\begin{array}{rll}
\displaystyle -\Delta_p u_\infty-\frac{1}{u_\infty^\delta}&=f(x,u_\infty) \;&\displaystyle \mbox{in }\,\Omega,\\
\displaystyle u_\infty&=0\; &\displaystyle \mbox{on }\;\partial\Omega.
\end{array}\right.
\end{eqnarray*}
\end{theorem}
Using a time discretization method, Theorem \ref{1}, energy estimates and the weak comparison principle (see {\sc Cuesta-Tak\'a\v{c}} \cite{CuTa}, {\sc Fleckinger-Tak\'a\v{c}} \cite{FlTa}), we prove the following
\begin{theorem}\label{3}
 Let $0<\delta<2+\frac{1}{p-1}$, $h\in L^\infty(Q_T)$ and  $u_0\in W^{1,p}_0(\Omega)\cap {\mathcal C}$. Then there exists a unique weak solution $u$ to 
\begin{equation}
\label{linearheateq}
\left\{
\begin{aligned}
&u_t-\Delta_p u =\frac{1}{u^\delta}+h(x,t) \;\text{ in }\,Q_T,\\
&u =0\,\text{ on } \;\Sigma_T,\quad u>0 \text{ in }\, Q_T,\\
& u(0,x) =u_0(x)\;\text{ in }\Omega,
\end{aligned}
\right.
\end{equation}
such that $u(t)\in\, {\mathcal C}$ uniformly for $t\in\,[0,T]$. Moreover, $u$ belongs to $C([0,T],W^{1,p}_0(\Omega))$ and satisfies for any $t\in [0,T]$:
\begin{eqnarray}
\label{secondenergyeq}
\int_0^t\int_{\Omega}\big{(}\frac{\partial u}{\partial t}\big{)}^2{\rm d}x{\rm d}s&+&\frac{1}{p}\int_{\Omega}|\nabla u(t)|^p{\rm d}x-\frac{1}{1-\delta}\int_{\Omega}u^{1-\delta}(t){\rm d}x\\
=\int_0^t\int_{\Omega}h\frac{\partial u}{\partial t}{\rm d}x{\rm d}s&+&\frac{1}{p}\int_{\Omega}|\nabla u_0|^p{\rm d}x-\frac{1}{1-\delta}\int_{\Omega}u_0^{1-\delta}{\rm d}x\nonumber.
\end{eqnarray} 
\end{theorem}
\begin{remark}
Saying that $u(t)\in\, {\mathcal C}$ uniformly for $t\in\,[0,T]$ means that there exists $\underline{u}$, $\overline{u}\in \mathcal{C}$ such that $\underline{u}(x)\leq u(t,x)\leq \overline{u}(x)$ a.e $(x,t)\in \Omega\times [0,T]$.
\end{remark}
\begin{remark}
By Theorem \ref{2}, the restriction $\delta<2+\frac{1}{p-1}$ is sharp.
\end{remark}
Moreover, we have the following regularity result for solutions to \eqref{linearheateq} which is obtained from the theory of nonlinear monotone operators of \cite{Ba}:
\begin{proposition}
\label{regsobolev}
Assume that hypotheses of Theorem~\ref{3} are satisfied and set: 
\begin{equation}
\label{def-domaine-op}
{\mathcal D}(A)\eqdef\left\{ v\in {\mathcal C}\cap W^{1,p}_0(\Omega)\,|\,Av\eqdef-\Delta_p v-\frac{1}{v^\delta}\in L^\infty(\Omega)\right\}.
\end{equation}
If  in addition $u_0\in \overline{\mathcal{D}(A)}^{L^\infty(\Omega)}$, then the solution $u$ to \eqref{linearheateq} belongs to $C([0,T];C_0(\overline{\Omega}))$ and satisfies
\begin{enumerate}
\item[(i)] if $v$ is another weak solution to $(S_t)$ with initial datum $v_0\in\overline{\mathcal{D}(A)}^{L^\infty(\Omega)} $ and nonhomogeneous term $k\in L^\infty(Q_T)$ then the following estimate holds:
\begin{equation}
\|u(t)-v(t)\|_{L^\infty(\Omega)}\leq \|u_0-v_0\|_{L^\infty(\Omega)}+\int_0^t\|h(s)-k(s)\|_{L^\infty(\Omega)} {\rm d}s,\quad 0\leq t\leq T.\label{estimation}
\end{equation}
\item[(ii)] if $u_0\in {\mathcal D}(A)$ and $h\in W^{1,1}(0,T;L^\infty(\Omega))$ then $u\in W^{1,\infty}(0,T;L^\infty(\Omega))$ and $\Delta_p u+u^{-\delta}\in L^\infty(Q_T)$, and the following estimate holds:
\begin{equation}
\left \|\frac{{\rm d} u(t)}{{\rm d}t}\right \|_{L^\infty(\Omega)}\leq \|\Delta_p u_0+u_0^{-\delta}+h(0)\|_{L^\infty(\Omega)}+\int_0^T\left \|\frac{{\rm d} h(t)}{{\rm d}t}\right \|_{L^\infty(\Omega)}{\rm d}\tau.\label{estimationbiss}
\end{equation}
\end{enumerate}
\end{proposition}
Concerning problem $({\rm P}_t)$, we have the following
\begin{theorem}\label{3bis}
 Let $0<\delta<2+\frac{1}{p-1}$. Assume that $f$ is a bounded below Caratheodory function, and that $f$ is locally Lipschitz with respect to the second variable uniformly in $x\in\Omega$ and satisfying \eqref{sublineargrowth}. Let $u_0\in W^{1,p}_0(\Omega)\cap {\mathcal C}$. Then, for any $T>0$, there exists a unique weak solution, $u$, to $({\rm P_t})$ 
such that $u(t)\in\, {\mathcal C}$ uniformly for $t\in\,[0,T]$, $u\in C([0,T],W^{1,p}_0(\Omega))$ and $u$ satisfies for any $t\in [0,T]$:
\begin{eqnarray}
\label{secondenergyeqbis}
\int_0^t\int_{\Omega}\big{(}\frac{\partial u}{\partial t}\big{)}^2{\rm d}x{\rm d}s&+&\frac{1}{p}\int_{\Omega}|\nabla u(t)|^p{\rm d}x-\frac{1}{1-\delta}\int_{\Omega}u^{1-\delta}(t){\rm d}x\\
=\int_{\Omega}F(x,u(t)){\rm d}x+\frac{1}{p}\int_{\Omega}|\nabla u_0|^p{\rm d}x&-&\frac{1}{1-\delta}\int_{\Omega}u_0^{1-\delta}{\rm d}x-\int_{\Omega}F(x,u_0){\rm d}x,\nonumber
\end{eqnarray}
where $F(x,w)\eqdef\int_0^wf(x,s){\rm d}s$. 
\end{theorem}
A straightforward application of Proposition~\ref{regsobolev} yields the following
\begin{proposition}
\label{regsobolevbis}
Assume that conditions in Theorem~\ref{3bis} are satisfied. If in addition $u_0\in \overline{\mathcal{D}(A)}^{L^\infty(\Omega)}$, then the solution $u$ to $({\rm P_t})$  belongs to $C([0,T];C_0(\overline{\Omega}))$ and 
\begin{enumerate}
\item[(i)] there exists $\omega>0$ such that if $v$ is another weak solution to $(P_t)$ with initial datum $v_0\in \overline{\mathcal{D}(A)}^{L^\infty(\Omega)}$ then the following estimate holds:
$$
\|u(t)-v(t)\|_{L^\infty(\Omega)}\leq e^{\omega t} \|u_0-v_0\|_{L^\infty(\Omega)},\quad 0\leq t\leq T.
$$
\item[(ii)]  if $u_0\in {\mathcal D}(A)$ then $u\in W^{1,\infty}(0,T;L^\infty(\Omega))$ and $\Delta_p u+u^{-\delta}\in L^\infty(Q_T)$, and the following estimate holds:
$$
\left \|\frac{{\rm d} u(t)}{{\rm d}t}\right \|_{L^\infty(\Omega)}\leq e^{\omega t} \|\Delta_p u_0+u_0^{-\delta}+f(x,u_0)\|_{L^\infty(\Omega)}
$$
\end{enumerate}
\end{proposition}
\begin{remark}
The constant $\omega$ given above is equal to the Lipschitz constant of $f(x,\cdot)$ on $[\underline{u}, \overline{u}]$ where $\underline{u}$ and $\overline{u}$ are respectively subsolution and supersolution to $(Q)$ given in \eqref{subsol} and \eqref{supsol} below.
\end{remark}
From Theorems~\ref{3bis} and \ref{2bis}, we can show the following asymptotic behaviour for solutions to $({\rm P}_t)$:
\begin{theorem}
\label{3ter}
Let hypothesis in Theorem~\ref{3bis} satisfied and assume that $\frac{f(x,s)}{s^{p-1}}$ is decreasing in $(0,\infty)$ for a.e. $x\in \Omega$. Then, the solution to $({\rm P}_t)$ is defined in $(0,\infty)\times \Omega$ and satisfies
\begin{equation}
\label{convasympt}
u(t)\to u_\infty\quad\mbox{in}\;\;L^\infty(\Omega)\quad\mbox{as}\;\; t\to\infty,
\end{equation}
where $u_\infty$ is defined in Theorem~\ref{2bis}.
\end{theorem}
Concerning the non degenerate case, i.e $p=2$, we can give additional results. In particular, we prove the existence of solutions in the sense of distributions for any $0<\delta$. Precisely,
\begin{theorem}\label{4}
Let $0<\delta$ and $p=2$. Let $f$ satisfy assumptions in Theorem~\ref{3bis} and $u_0\in {\mathcal C}$. Then, for any $T>0$, there exists a unique solution $u\in C([0,T], L^2(\Omega))\cap L^\infty(Q_T)$ to $({\rm P}_t)$ in the sense of distributions, that is $u\in {\mathcal C}$ uniformly in $t\in [0,T]$ and for any $\phi\in {\mathcal D}(Q_T)$, we have
\begin{equation}
\label{solutionp=2}
-\int_{Q_T}u\frac{\partial\phi}{\partial t}\,{\rm d}x{\rm d}t-\int_{Q_T}u\Delta\phi\,{\rm d}x{\rm d}t=\int_{Q_T}\left(\frac{1}{u^\delta}+ f(x,u)\right)\phi\,{\rm d}x{\rm d}t.
\end{equation}
In addition, we have for $0<\eta$ small enough, the following regularity property:
\begin{itemize}
\item[(i)] if $\delta<\frac{1}{2}$ and $u_0\in {\mathcal C}\cap H^{2-\eta}(\Omega)$, then $u\in C\big{(}[0,T];\,H^{2-\eta}(\Omega)\big{)}$;
\item[(ii)] if $\frac{1}{2}\leq \delta<1$ and $u_0\in {\mathcal C}\cap H^{\frac{5}{2}-\delta-\eta}(\Omega)$, then $u\in C\big{(}[0,T],H^{\frac{5}{2}-\delta-\eta}(\Omega)\big{)}$;
\item[(iii)] if $1\leq \delta$ and $u_0\in {\mathcal C}\cap H^{\frac{1}{2}+\frac{2}{\delta+1}-\eta}(\Omega)$, then $u\in C\big{(}[0,T],H^{\frac{1}{2}+\frac{2}{\delta+1}-\eta}(\Omega)\big{)}$.
\end{itemize}
Moreover, 
\begin{enumerate}
\item[(iv)] if $u_1$, $u_2$ are solutions corresponding to initial data $u_{1,0}\in  \mathcal{C}$, $u_{2,0}\in  \mathcal{C}$ respectively, then there exist $\underline{u}$,  $\overline{u}$ in $\mathcal{C}$ and a positive constant $\omega$ (proportional to the lipschitz constant of $f(x,\cdot)$ in $[0,\|\overline{u}\|_{\infty}]$) such that
\begin{equation}
\| u_1(t)-u_2(t)\|_{L^2(\Omega)}\leq e^{(\omega-\lambda_1) t} \| u_{1,0}-u_{2,0}\|_{L^2(\Omega)}\; \mbox{ and }\; \underline{u}\leq u^i\leq  \overline{u}, \; i=1,2.\label{eqtomega}
\end{equation}
\item[(v)] if $f(x,\cdot)$ is a nonincreasing function then \eqref{eqtomega} is true with $\omega=0$. Then, the solution to $({\rm P}_t)$, $u$, defined in $(0,\infty)$ satisfies $u(t)\to u_\infty$ as $t\to +\infty$ in $L^2(\Omega)$ where $u_\infty$ is the solution given in Theorem~\ref{2bis}.
\end{enumerate}
\end{theorem}
\begin{remark}
In particular, if $\delta<3$ and $u_0\in H^1_0(\Omega)$, then we recover $u\in C([0,T];\,H^1_0(\Omega))$. Note also that for arbitrary $\delta>0$ there is $\epsilon>0$ such that $u\in C([0,T];\,H^{\frac{1}{2}+\epsilon}_0(\Omega))$ if $u_0\in H^{\frac{1}{2}+\epsilon}_0(\Omega)$.
\end{remark}
Theorem \ref{4} is established using the interpolation theory in Sobolev spaces and the $L^p-L^q$-maximal regularity results of the linear heat equation. Under the assumptions given in Theorem~\ref{2bis}, we can derive from Theorem~\ref{4} some stabilization properties. Precisely, we prove
\begin{theorem}\label{5}
Let $p=2$, $\delta<3$, $u_0\in {\mathcal C}\cap H^1_0(\Omega)$. Assume that $f$ satisfies assumptions of Theorem~\ref{3ter}. Then, the solution to $({\rm P}_t)$, $u$, defined in $(0,\infty)\times\Omega$ satisfies $u(t)\to u_\infty$ as $t\to +\infty$ in $L^\infty(\Omega)\cap H^1_0(\Omega)$ where $u_\infty$ is the solution given in Theorem~\ref{2bis}.
\end{theorem}
We give now briefly the state of art concerning parabolic quasilinear singular equations. The corresponding stationary equation was studied profusely in the litterature. In particular the case $p=2$, mostly when $\delta<1$ and also when $g$ depends on $u$ was considered in detail (see the pionniering work {\sc Crandall-Rabinowitz-Tartar} \cite{CrRaTa}, the bibliography in {\sc Hern\'andez-Mancebo} \cite{HeMa} and {\sc Perera-Silva} \cite{PeSi}). The case $p\neq 2$ was not considered so far. We can mention the work {\sc Aranda-Godoy} \cite{ArGo} where existence results are obtained via the bifurcation theory for $1<p\leq 2$ and $g=g(u)$ satisfying some growth conditions. In {\sc Giacomoni-Schindler-Tak\'a\v{c}} \cite{GiScTa} the existence and multiplicity  results (for $1<p<\infty$ $g(u)=u^q$ with $p-1<q\leq p^*-1$ and $0<\delta<1$ are proved by using variational methods and regularity results in H\"older spaces. Concerning the parabolic case, avalaible results mostly concern 
 the case $p=2$. In this regard, we can quote the result in {\sc Hernandez-Mancebo-Vega} \cite{HeMaVe} where in the range $0<\delta<\frac{1}{2}$, properties of the linearised operator (in $C^1_0(\overline{\Omega})$) and the validity of the strong maximum principle are studied. In {\sc Tak\'a\v{c}} \cite{Ta}, a stabilization result in $C^1$ is proved for a class of parabolic singular problems via a clever use of weighted Sobolev spaces. We also mention the work {\sc Davila-Montenegro} \cite{DaMo} still concerning the case $p=2$ and with singular absorption term. In this nice work, the authors achieved uniqueness within the class of functions satisfying $u(x,t)\geq c\dist(x,\partial\Omega)^\gamma$ for suitable $\gamma$ and $c>0$ and discuss the asymptotic behaviour of solutions. Finally, we would like to quote the nice paper {\sc Winkler} \cite{Wi} where the author shows that uniqueness is violated in case of non homogeneous boundary Dirichlet condition.
\par The present paper is organized as follows. The two next sections (Section~\ref{linearcase}, Section~\ref{linearparabolic}) contain the proofs of Theorems~\ref{1}, \ref{2}, \ref{3} and Proposition~\ref{regsobolev}.
Theorems~\ref{2bis}, \ref{3bis}, \ref{3ter} and Proposition~\ref{regsobolevbis} are established in section~\ref{nonlinearcase}. Finally, the non degenerate case (i.e. $p=2$) is dealt in Section~\ref{pegal2} where in particular Theorems~\ref{4} and \ref{5} are proved.
\section{Proof of Theorems \ref{1} and \ref{2}}\label{linearcase}
We first prove Theorem \ref{1}. 
\begin{proof}
First, let us consider the case $\delta<1$. For $\lambda>0$, we define the following energy functional~:
\begin{eqnarray*}
E_\lambda(u)\eqdef\frac{1}{2}\int_{\Omega}u^2{\rm d}x+\frac{\lambda}{p}\int_{\Omega}|\nabla u|^p{\rm d}x-\frac{\lambda}{1-\delta}\int_{\Omega}(u^+)^{1-\delta}{\rm d}x-\int_{\Omega}gu\,{\rm d}x.
\end{eqnarray*}
$E_\lambda$ is well-defined in $X=W^{1,p}_0(\Omega)$ if $p\geq \frac{2{\rm N}}{{\rm N}+2}$. If $1<p<\frac{2{\rm N}}{{\rm N}+2}$, $E_\lambda$ is well-defined in $X=W^{1,p}_0(\Omega)\cap L^2(\Omega)$. It is easy to see that $E_\lambda$ is stricly convex, continuous and coercive in $X$. Thus, since $X$ is reflexive, $E_\lambda$ admits a unique global minimizer denoted by $u_\lambda$.
We show now that $u_\lambda\in {\mathcal C}$. Let $\phi_1$ be the normalized positive eigenfunction associated with the principal eigenvalue $\lambda_1(\Omega)$ of $-\Delta_p$ with homogeneous boundary Dirichlet conditions (see {\sc Anane} \cite{Anane-1}, \cite{Anane-2} for further details):
\begin{equation}
\label{EV:phi_1}
  - \Delta_p \phi_1 = \lambda_1\, |\phi_1|^{p-2} \phi_1
    \quad\mbox{ in }\, \Omega ;\qquad
  \phi_1 = 0 \quad\mbox{ on }\, \partial\Omega ,
\end{equation}
$\phi_1\in W_0^{1,p}(\Omega)$
is normalized by $\phi_1 > 0$ in $\Omega$ and
$\int_{\Omega} \phi_1^p \,\mathrm{d}x = 1$.
Note that the strong maximum and boundary point principles from
{\sc V\'az\-quez} \cite[Theorem~5, p.~200]{Va} guarantee
$\phi_1 > 0$ in $\Omega$ and
$\frac{\partial\phi_1}{\partial\nu} < 0$ on $\partial\Omega$,
respectively.
Hence, since $\phi_1\in C^1(\overline{\Omega})$, there are constants
$\ell$ and $L$, $0 < \ell < L$, such that
$\ell\, d(x)\leq \phi_1(x)\leq L\, d(x)$ for all $x\in \Omega$. Moreover, we observe that for $\epsilon>0$ small enough (depending on $\lambda$, $\delta$ and $g$) we have 
\begin{equation}
\left\{\begin{array}{rll}
\displaystyle \epsilon\phi_1-\lambda(\Delta_p(\epsilon\phi_1)+\frac{1}{(\epsilon\phi_1)^\delta})&\displaystyle <g&\displaystyle \mbox{in}\;\Omega,\\
\displaystyle \epsilon\phi_1&\displaystyle =0\;&\displaystyle \mbox{on }\;\partial\Omega.
\end{array}
\right.\label{eq1}
\end{equation}
Thus, for $t>0$, we set $v_\lambda\eqdef(\epsilon\phi_1-u_\lambda)^+$ and $\chi(t)\eqdef E_\lambda(u_\lambda+tv_\lambda)$. From the Hardy Inequality, it follows that $\chi$ is differentiable for $t\in (0,1]$ and
\begin{eqnarray*}
\chi'(t)=<E'_\lambda(u_\lambda+tv_\lambda),v_\lambda>.
\end{eqnarray*}
The optimality of $u_\lambda$ guarantees $\chi'(0)=0$ and the strict convexity of $E_\lambda$ ensures that $t\rightarrow\,\chi'(t)$ is increasing. Therefore, with \eqref{eq1} we obtain that
\begin{eqnarray*}
0<\chi'(1)=<E'_\lambda(\epsilon\phi_1),v_\lambda><0
\end{eqnarray*}
if $v_\lambda$ has non-zero measure support. Then $\epsilon\phi_1\leq u_\lambda$ and $E_\lambda$ is G\^ateaux-differentiable in $u_\lambda$. Consequently, for any $\phi\in X$,
\begin{eqnarray*}
<E'_\lambda(u_\lambda),\phi>=<u_\lambda-\lambda(\Delta_pu_\lambda+\frac{1}{u_\lambda^\delta})-g,\phi>=0.
\end{eqnarray*}
We observe that if $\delta<2+\frac{1}{p-1}$ then
\begin{equation}
\label{monotonicity2}
u\to u-\lambda(\Delta_p u-\frac{1}{u^\delta})\quad\mbox{is monotone from}\; W^{1,p}_0(\Omega)\cap {\mathcal C}\;\mbox{to}\; W^{-1,\frac{p}{p-1}}(\Omega).
\end{equation}
Then, by the weak comparison principle, we have also that
\begin{eqnarray*}
u_\lambda\leq M
\end{eqnarray*}
for any $M>|g|_{L^\infty(\Omega)}+\frac{\lambda}{|g|_{L^\infty(\Omega)}^\delta}$. Then, $u_\lambda\in L^\infty(\Omega)$. Let $U\in C^{1,\alpha}(\overline{\Omega})\cap \mathcal{C}$ (with suitable $0<\alpha<1$) be the unique positive solution (see {\sc Giacomoni-Schindler-Tak\'a\v{c}} \cite[Theorem~B.1]{GiScTa} for the existence and regularity of $U$) to
\begin{eqnarray}\label{eqhom}
\left\{\begin{array}{rlr}
\displaystyle  -\Delta_p u&\displaystyle =\frac{1}{u^\delta}&\mbox{in}\;\Omega\displaystyle ,\\
\displaystyle u&\displaystyle =0 \;&\displaystyle \mbox{on }\, \partial\Omega.
\end{array}
\right.
\end{eqnarray}
Therefore, observing that for $M'>0$,
\begin{eqnarray*}
\left\{\begin{array}{rll}
\displaystyle  M'U-\lambda\left(\Delta_p (M'U)+\frac{1}{(M'U)^\delta}\right)&\displaystyle =M'U+\frac{\lambda({M'}^{p-1}-{M'}^{-\delta})}{U^\delta}\quad&\displaystyle\mbox{in}\;\Omega\\
\displaystyle M'U&\displaystyle =0 \;&\displaystyle \mbox{on }\, \partial\Omega,
\end{array}
\right.
\end{eqnarray*}
and by the weak comparison principle, we get that $u_\lambda\leq M'U$ for $M'$ large enough. Together with $\epsilon\phi_1\leq u_\lambda$, it follows that $u_\lambda\in {\mathcal C}$. Again using {\sc Giacomoni-Schindler-Tak\'a\v{c} \cite[Theorem~B.1]{GiScTa}}, we get that $u_\lambda\in C^{1,\alpha}(\overline{\Omega})$ and then $u_\lambda\in C_0(\overline{\Omega})$.
\par We consider now the case $\delta\geq 1$. We use in this case the weak comparison principle, the existence of suitable subsolutions and supersolutions of the following approximated problem~:
\begin{eqnarray*}
({\rm P}_\epsilon)\left\{\begin{array}{lc}
& u-\lambda\left(\Delta_p u+\frac{1}{(u+\epsilon)^\delta}\right)=g\quad\mbox{in}\;\Omega,\\
& u=0 \;\mbox{ on }\, \partial\Omega,\;\; u>0\;\mbox{in}\;\Omega.
\end{array}
\right.
\end{eqnarray*}
Using a minimization argument as in the case $\delta<1$, we get the existence and the uniqueness of the solution to $({\rm P}_\epsilon)$, denoted $u_\epsilon$, in $W^{1,p}_0(\Omega)^+\cap L^\infty(\Omega)$. From the elliptic regularity theory (see {\sc lieberman} \cite{Li}), we obtain that $u_\epsilon\in C^{1,\alpha}(\overline{\Omega})$ for some $\alpha\in (0,1)$. We now construct appropriate subsolutions and supersolutions for $({\rm P}_\epsilon)$.
For $\delta =1$, by straightforward computations we have that for $A>0$ large enough (depending on the diameter of $\Omega$), and for $\eta>0$ small enough (depending on $\lambda$ and $g$ but not on $\epsilon$)
\begin{eqnarray}
\label{epsisub}
\underline{u}_\epsilon\eqdef(\eta\phi_1+\epsilon')\left[\ln(\frac{A}{\eta\phi_1+\epsilon'})\right]^{\frac{1}{p}}-\epsilon'\left[\ln(\frac{A}{\epsilon'})\right]^{\frac{1}{p}},
\end{eqnarray}
with $\epsilon'>0$ satisfying $\epsilon=\epsilon'[\ln(\frac{A}{\epsilon'})]^{\frac{1}{p}}$, is a subsolution to $({\rm P}_\epsilon)$. Similarly, for $M>0$ large enough (depending on $\lambda$ and $g$ but not on $\epsilon$)
\begin{eqnarray}
\label{epsisup}
\bar{u}_\epsilon\eqdef(M\phi_1+\epsilon')\left[\ln(\frac{A}{M\phi_1+\epsilon'})\right]^{\frac{1}{p}}-\epsilon'\left[\ln(\frac{A}{\epsilon'})\right]^{\frac{1}{p}},
\end{eqnarray}
is a supersolution to $({\rm P}_\epsilon)$ satisfying $\bar{u}_\epsilon\geq\underline{u}_\epsilon$. 
If $\delta >1$, we consider the following subsolution and supersolution respectively:
\begin{eqnarray}
\label{epsisub2}
\underline{u}_\epsilon\eqdef\eta\left[(\phi_1+\epsilon^{\frac{p-1+\delta}{p}})^{\frac{p}{p-1+\delta}}-\epsilon\right],
\end{eqnarray}
for $\eta >0$ small enough and 
\begin{eqnarray}
\label{epsisup2}
\bar{u}_\epsilon\eqdef M\left[(\phi_1+\epsilon^{\frac{p-1+\delta}{p}})^{\frac{p}{p-1+\delta}}-\epsilon\right],
\end{eqnarray}
for $M>0$ large enough.
Since the operator $u\rightarrow\,-\Delta_pu-\frac{1}{(u+\epsilon)^\delta}$ is monotone from $(W^{1,p}_0(\Omega))^+$ to $W^{-1,\frac{p}{p-1}}(\Omega)$ (see {\sc Deimling} \cite{De} for further details about the theory of monotone operators), we get from the weak comparison principle that 
\begin{equation}\label{sub}
\underline{u}_\epsilon\leq u_\epsilon\leq \bar{u}_\epsilon.
\end{equation}
Again from the weak comparison principle, we have that
\begin{eqnarray*}
0<\epsilon_1<\epsilon_2\Rightarrow\,\left\{\begin{array}{cl}
 \displaystyle u_{\epsilon_2}<u_{\epsilon_1}\quad&\displaystyle \mbox{in}\;\Omega,\\
 \displaystyle u_{\epsilon_1}+\epsilon_1<u_{\epsilon_2}+\epsilon_2\quad&\displaystyle \mbox{in}\;\Omega,
\end{array}
\right.
\end{eqnarray*}
from which it follows that $(u_{\epsilon_n})_{n\in \N}$ is a Cauchy sequence as $\epsilon_n\to 0^+$ in $C_0(\bar{\Omega})$. Then $u_{\epsilon_n}\to u$ in $C_0(\bar{\Omega})$ and by passing to the limit in \eqref{sub}
we deduce that $\underline{u}\leq u\leq \overline{u}$ where $\underline{u}$ and $\overline{u}$ are the respective subsolution and supersolution to $(P)$ given by
\begin{eqnarray}
\quad \phi =\left\{\begin{array}{ll}
\phi_1\left(\ln(\frac{A}{\phi_1})\right)^{\frac{1}{p}}\quad&\mbox{if}\; \delta=1,\\
\phi_1^{\frac{p}{p-1+\delta}}\quad&\mbox{if}\;\delta>1,
\end{array}\right.\mbox{ and }\quad 
\displaystyle \left\{\begin{array}{ll}
\displaystyle \underline{u}=\eta \phi,\\
\displaystyle \overline{u}=M \phi,
\end{array}\right.
\label{subsol-0}
\end{eqnarray}
(with $A$, $M>0$ large enough and $\eta >0$ small enough, depending on $\lambda$, $g$). Then it follows that $u\in{\mathcal C}\cap C_0(\overline{\Omega})$. Let us show that $u$ is a weak solution to  $({\rm P})$. Since $\delta<2+\frac{1}{p-1}$, we get from \eqref{sub} and the Hardy Inequality that 
\begin{eqnarray*}
\displaystyle\limsup_{n\in\N}\int_{\Omega}\frac{u_{\epsilon_n}}{(u_{\epsilon_n}+\epsilon_n)^\delta}{\rm d}x<+\infty,
\end{eqnarray*}
 and consequently, by multiplying by $u_{\epsilon_n}$ the first equation of $({\rm P}_\epsilon)$ and integrating by parts, we obtain $\displaystyle\sup_{n\in\N}\|u_{\epsilon_n}\|_{W^{1,p}_0(\Omega)}<+\infty$.  Moreover, by subtracting $({\rm P}_{\epsilon_n})$ to $({\rm P}_{\epsilon_m})$ and recalling the following  well-know inequality for $p\geq 2$, $w,v$ in $W^{1,p}(\Omega)$ and suitable $C_1>0$,
\begin{equation}\label{inegaliteAlg}
\int_\Omega (|\nabla w|^{p-2}\nabla w-|\nabla v|^{p-2}\nabla v)\nabla (w-v){\rm d}x \geq C_1\int_\Omega |\nabla (w-v)|^p {\rm d}x 
\end{equation}
and the following well-know inequality for $p< 2$, $w,v$ in $W^{1,p}(\Omega)$ and suitable $C_2>0$,
\begin{equation}\label{inegaliteAlg2}
\int_\Omega (|\nabla w|^{p-2}\nabla w-|\nabla v|^{p-2}\nabla v)\nabla (w-v){\rm d}x \geq \frac{C_2\left(\int_\Omega |\nabla (w-v)|^p{\rm d}x\right)^{\frac{2}{p}}}{\left(\left(\int_\Omega |\nabla w|^p{\rm d}x \right)^{\frac{1}{p}}+\left (\int_\Omega |\nabla v|^p{\rm d}x \right)^{\frac{1}{p}}\right )^{2-p}},
\end{equation}
we obtain
\begin{equation}
\nonumber
<-\Delta_p u_{\epsilon_n}+\Delta_p u_{\epsilon_m},u_{\epsilon_n}-u_{\epsilon_m}>\geq \left\{\begin{array}{ll}
C_1\|u_{\epsilon_n}-u_{\epsilon_m}\|_{W^{1,p}_0(\Omega)}^p&\mbox{if}\; p\geq 2,\\
C_2\frac{\|u_{\epsilon_n}-u_{\epsilon_m}\|_{W^{1,p}_0(\Omega)}^2}{(\|u_{\epsilon_n}\|_{W^{1,p}_0(\Omega)}+\|u_{\epsilon_m}\|_{W^{1,p}_0(\Omega)})^{2-p}}&\mbox{if}\; p<2.
\end{array}
\right.
\end{equation}
Then we deduce that $u_{\epsilon_n}$ is also a Cauchy sequence in $W^{1,p}_0(\Omega)$ as $\epsilon_n\to 0^+$ and that $u_{\epsilon_n}\to u$ in $W^{1,p}_0(\Omega)$. Thus, it is easy to derive that $u$ is a weak solution to $({\rm P})$. Finally, the uniqueness of the solution to  $({\rm P})$ in $W^{1,p}_0(\Omega)\cap {\mathcal C}$ follows from \eqref{monotonicity2}.
\end{proof} 
We prove now Theorem \ref{2}.
\begin{proof}
Let $\delta\geq 2+\frac{1}{p-1}$. We give an alternative proof for existence of solutions.
Let $(\Omega_k)_k$ be an increasing  sequence of smooth domains such that $\Omega_k\uparrow\Omega$ (in the Hausdorff Topology) and $\frac{1}{k}\leq\dist(x,\partial\Omega) \leq \frac{2}{k}$, $\forall\, x\in \Omega_k$.
We use the sub\--solution and super\--solution technique in $\Omega_k$ and pass to the limit as $k\to\infty$. For $0<\eta<M$, let 
\begin{eqnarray*}
\underline{u}\eqdef\eta(\phi_1)^{\frac{p}{p-1+\delta}},\quad \bar{u}\eqdef M(\phi_1)^{\frac{p}{p-1+\delta}}.
\end{eqnarray*}
For $\eta$ small enough and $M$ large enough, $\underline{u}$ and $\bar{u}$ are respectively a subsolution and a supersolution  to $({\rm P})$ and both belong to ${\mathcal C}\cap C_0(\overline{\Omega})$. By using  a minimization argument in $W^{1,p}_0(\Omega_k)$ as in the case $\delta<1$ (note that the term associated to $\frac{1}{u^\delta}$ in the energy functional is not singular since $\underline{u}>0$ on $\partial\Omega_k$), there is a positive solution $v_k\in W^{1,p}_0(\Omega_k)$ to
\begin{eqnarray*}
\left\{\begin{array}{rll}
 \displaystyle u-\lambda(\Delta_p(u +\underline{u})+\frac{1}{(u+\underline{u})^\delta})&\displaystyle=g-\underline{u}\quad&\displaystyle\mbox{in}\;\Omega_k,\\
\displaystyle u&\displaystyle=0\;&\displaystyle\mbox{on }\;\partial\Omega_k. 
\end{array}
\right.
\end{eqnarray*}
From {\sc Lieberman} \cite{Li}, $v_k\in C^{1,\beta}(\overline{\Omega}_k)$ for some $\beta\in (0,1)$. Then, $u_k\eqdef u+\underline{u}\in C^{1,\beta}(\overline{\Omega}_k)$ satisfies
\begin{eqnarray*}
\left\{\begin{array}{rll}
 \displaystyle u_k-\lambda(\Delta_pu_k+\frac{1}{u_k^\delta})&\displaystyle =g\quad&\displaystyle \mbox{in}\;\Omega_k,\\
\displaystyle u&\displaystyle =\underline{u}\;&\displaystyle \mbox{on }\;\partial\Omega_k.
\end{array}
\right.
\end{eqnarray*}
and $\underline{u}\leq u_k\leq \bar{u}$ holds. From the weak comparison principle, we have that $u_k\leq u_{k+1}$ in $\Omega_k$, and if $\Tilde u_k\in C_0(\overline\Omega)$ denotes the extension of $u_k$ by $\underline{u}$ outside $\Omega_k$, then $\underline{u}\leq \Tilde u_k\leq \Tilde u_{k+1}\leq \bar{u}$ and by Dini's Theorem, $\Tilde u_k\to u$ in $C_0(\overline\Omega)\cap \mathcal{C}$. Moreover, for every compact subset ${\mathcal K}$ of $\Omega$ and $k$ large enough so that ${\mathcal K}\subset \Omega_k$, we have $\frac{1}{\Tilde{u}_k^\delta}=\frac{1}{u_k^\delta}\leq \frac{1}{\underline{u}^\delta}\in L^\infty({\mathcal K})$ and $ \Delta_p\Tilde u_k= \Delta_p u_k=-g+u_k-\frac{\lambda}{u_k^\delta}$ bounded in $L^\infty({\mathcal K})$ uniformely in $k$. Then using local regularity results (see for instance {\sc Serrin} \cite{Se}, {\sc Tolksdorf} \cite{To} and \cite{To1}, {\sc DiBenetdetto} \cite{Be}), for $k$ large enough we get that $u_k$ is bounded in $C^1({\mathcal K})$
  and then converges to $u$ in $W^{1,p}({\mathcal K})$. Then $\underline{u}\leq u\eqdef\displaystyle\lim_{k\to\infty}u_k\in W^{1,p}_{\rm loc}(\Omega)$ and satisfies $({\rm P})$ in the sense of distributions. 
Let us show that $u\not\in W^{1,p}_0(\Omega)$. For that, we argue by contradiction: assume that $u\in W^{1,p}_0(\Omega)$. Then, from the equation in $({\rm P})$, we get that $\frac{1}{u^\delta}\in W^{-1,\frac{p}{p-1}}(\Omega)$. Thus, $\int_{\Omega}{\bar{u}}^{1-\delta}{\rm d}x\leq \int_{\Omega}{u}^{1-\delta}{\rm d}x<+\infty$ which contradicts the definition of $\bar{u}$. The proof of Theorem \ref{2} is now complete.
\end{proof}
\section{Proof of Theorem \ref{3} and Proposition \ref{regsobolev}}\label{linearparabolic}
Using Theorem \ref{1} and a time discretization method, we prove Theorem \ref{3}.
\begin{proof}
Let $N\in \N^*$, $n\geq 2$ and $\Delta_t=\frac{T}{N}$. For $0\leq n\leq N$, we define $t_n\eqdef n\Delta_t$, $h^n(\cdot)\eqdef\frac{1}{\Delta_t}\int_{t_{n-1}}^{t_n}h(\tau,\cdot)\,{\rm d}\tau\,\in L^\infty(\Omega)$ and the function $h_{\Delta_t}\in L^\infty(Q_T)$ as follows:
$$
h_{\Delta_t}(t)\displaystyle \eqdef h^n, \quad \forall t\in\,[t_{n-1},t_n), \; \forall\,n\in \{1,\dots,N\}.
$$
Notice that we have for all $1<q<+\infty$:
\begin{eqnarray}
\|h_{\Delta_t}\|_{L^q(Q_T)}&\leq& (T|\Omega|)^{\frac{1}{q}}\|h\|_{\infty},\label{eq2}\\
h_{\Delta_t}&\to &h \mbox{ in } L^q(Q_T) \label{eq2bis}.
\end{eqnarray}
From Theorem \ref{1} (with $\lambda=\Delta_t$, $g=\Delta_t h^n+u^{n-1}\in L^\infty(\Omega)$), we define by iteration $u^n\in W^{1,p}_0(\Omega)\cap {\mathcal C}$ with the following scheme:
\begin{equation}\label{scheme}
\left\{\begin{array}{rll}
\displaystyle \frac{u^n-u^{n-1}}{\Delta_t}-\Delta_pu^n-\frac{1}{(u^n)^\delta}&=h^n\;&\mbox{in}\;\Omega,\\
\displaystyle u^n&=0\;&\mbox{on }\;\partial\Omega,
\end{array}
\right.
\end{equation}
and $u^0=u_0\in W^{1,p}_0(\Omega)\cap {\mathcal C}$. Then, defining functions $u_{\Delta_t}$, $\tilde{u}_{\Delta_t}$ by: for all $\,n\in \{1,\dots,N\}$,
\begin{equation}
\forall t\in\,[t_{n-1},t_n), \left \{\begin{array}{lrl}
& \displaystyle u_{\Delta_t}(t)&\eqdef u^n,\\
& \displaystyle \tilde{u}_{\Delta_t}(t)&\eqdef \displaystyle \frac{(t-t_{n-1})}{\Delta_t}(u^n-u^{n-1})+u^{n-1},
\end{array}
\right.\label{defUdiscret}
\end{equation}
we have that
\begin{eqnarray}\label{distributions}
\frac{\partial\tilde{u}_{\Delta_t}}{\partial t}-\Delta_p u_{\Delta_t}-\frac{1}{{u_{\Delta_t}}^\delta}=h_{\Delta_t}\in L^\infty(Q_T).
\end{eqnarray}
\par 
Using energy estimates, we first establish some apriori estimates for $u_{\Delta_t}$ and $\tilde{u}_{\Delta_t}$ independent of $\Delta_t$. Precisely, multiplying \eqref{scheme} by $\Delta_t u^n$ and summing from $n=1$ to $N'\leq N$, we get
 for $\epsilon>0$ small, by the Young Inequality and \eqref{eq2},
\begin{eqnarray}
\sum_{n=1}^{N'}\int_{\Omega}(u^n-u^{n-1})u^n{\rm d}x+\Delta_t\left[\displaystyle\sum_{n=1}^{N'}\|u^n\|_{W^{1,p}_0(\Omega)}^p-\displaystyle\sum_{n=1}^{N'}\int_{\Omega}(u^n)^{1-\delta}{\rm d}x\right]&\leq& \nonumber \\
C(\epsilon)T|\Omega|\|h\|_{\infty}^{\frac{p}{p-1}}+\epsilon \Delta_t\displaystyle\sum_{n=1}^{N'}\|u^n\|^p_{W^{1,p}_0(\Omega)}.&&\label{esti2}
\end{eqnarray}
In addition,
\begin{eqnarray}
&\displaystyle\sum_{n=1}^{N'}\int_{\Omega}(u^n-u^{n-1})u^n{\rm d}x=\frac{1}{2}\displaystyle\sum_{n=1}^{N'}\int_{\Omega}(|u^n-u^{n-1}|^2+|u^n|^2-|u^{n-1}|^2){\rm d}x=\nonumber\\
&\displaystyle \frac{1}{2}\displaystyle\sum_{n=1}^{N'}\int_{\Omega}|u^n-u^{n-1}|^2{\rm d}x+\frac{1}{2}\int_{\Omega}|u^{N'}|^2{\rm d}x-\frac{1}{2}\int_{\Omega}|u_0|^2{\rm d}x.\label{esti4}
\end{eqnarray}
Next, we estimate the singular term in the above expression. For that, arguing as in the proof of Theorem \ref{1}, we can prove the existence of $\underline{u},\,\bar{u}\in W^{1,p}_0(\Omega)\cap {\mathcal C}$ such that $\underline{u}\leq u_0\leq \bar{u}$ (since $u_0\in {\mathcal C}$) and such that 
\begin{eqnarray*}
-\Delta_p \underline{u}-\frac{1}{\underline{u}^\delta}\leq -\|h\|_{L^\infty(Q_T)}\quad\mbox{in}\;\Omega,
\end{eqnarray*}
\begin{eqnarray*}
-\Delta_p \bar{u}-\frac{1}{\bar{u}^\delta}\geq \|h\|_{L^\infty(Q_T)}\quad\mbox{in}\;\Omega.
\end{eqnarray*}
Indeed, if $\delta<1$ choose $\underline{u}=\eta\phi_1$ and $\overline{u}=M U$ with $U$ solution of \eqref{eqhom} and  if $\delta\geq 1$ choose $\underline{u}$, $\overline{u}$ as in \eqref{subsol-0}, where $A>0$, $M>0$ are large enough and $\eta >0$ is small enough. Note that $A$, $M$, $\eta$ depend on $\|h\|_{L^\infty(Q_T)}$.
Then iterating the application of the weak comparison principle, we obtain that for all $n\in\N$, $\underline{u}\leq u^n\leq \bar{u}$ which implies that 
\begin{equation}\label{controluniform}
\underline{u}\leq u_{\Delta_t},\,\tilde{u}_{\Delta_t}\leq \bar{u}.
\end{equation}
Therefore, since $\delta<2+\frac{1}{p-1}$,
\begin{equation}
\label{esti1}
\Delta_t\displaystyle\sum_{n=1}^{N'}\int_{\Omega}(u^n)^{1-\delta}{\rm d}x\leq \left\{\begin{aligned}
T\int_{\Omega}\bar{u}^{1-\delta}{\rm d}x<+\infty \quad\mbox{if}\; \delta\leq 1,\\
T\int_{\Omega}\underline{u}^{1-\delta}{\rm d}x<+\infty \quad\mbox{if}\; \delta >1.\\
\end{aligned}
\right.
\end{equation}
Gathering the \eqref{esti2}, \eqref{esti4}, \eqref{controluniform} and \eqref{esti1}, we get that $u_{\Delta_t},\,\tilde{u}_{\Delta_t}\in{\mathcal C}$ uniformly and are bounded in $L^p(0,T;\,W^{1,p}_0(\Omega))\cap L^\infty(0,T;\,L^\infty(\Omega))$. We now use a second energy estimate. Multiplying \eqref{scheme} by $u^n-u^{n-1}$ and summing from $n=1$ to $N'\leq N$,
we get by the Young Inequality
\begin{eqnarray}
\Delta_t\displaystyle\sum_{n=1}^{N'}\int_{\Omega}\big{(}\frac{u^n-u^{n-1}}{\Delta_t}\big{)}^2{\rm d}x&+&\displaystyle\sum_{n=1}^{N'}\int_{\Omega}|\nabla u^n|^{p-2}\nabla u^n\cdot\nabla(u^n-u^{n-1}){\rm d}x\nonumber\\
-\displaystyle\sum_{n=1}^{N'}\int_{\Omega}\frac{u^n-u^{n-1}}{(u^n)^\delta}{\rm d}x&\leq&\frac{\Delta_t}{2}\displaystyle\sum_{n=1}^{N'}\left[\int_{\Omega}(h^n)^2{\rm d}x+\int_{\Omega}(\frac{u^n-u^{n-1}}{\Delta_t})^2{\rm d}x\right]\label{seconden}
\end{eqnarray}
which implies that
\begin{eqnarray}
\frac{\Delta_t}{2}\displaystyle\sum_{n=1}^{N'}\int_{\Omega}(\frac{u^n-u^{n-1}}{\Delta_t})^2{\rm d}x&+&\displaystyle\sum_{n=1}^{N'}\int_{\Omega}|\nabla u^n|^{p-2}\nabla u^n\cdot\nabla(u^n-u^{n-1}){\rm d}x\nonumber \\
&-&\displaystyle\sum_{n=1}^{N'}\int_{\Omega}\frac{u^n-u^{n-1}}{(u^n)^\delta}{\rm d}x\leq|\Omega|\frac{T}{2}\|h\|^2_{L^\infty(Q_T)}\label{esti4b}.
\end{eqnarray}
From the convexity of the terms $\int_{\Omega}|\nabla u|^p{\rm d}x$ and $-\frac{1}{1-\delta}\int_{\Omega}u^{1-\delta}{\rm d}x$ we derive the following estimates:
\begin{equation}
\label{ineqconvexity}
\begin{aligned}
\frac{1}{p}\left[\int_{\Omega}|\nabla u^n|^p{\rm d}x-\int_{\Omega}|\nabla u^{n-1}|^p{\rm d}x\right]&\leq \int_{\Omega}|\nabla u^n|^{p-2}\nabla u^n\nabla(u^n-u^{n-1}){\rm d}x,\\
 \frac{1}{1-\delta}\left [\int_{\Omega}(u^{n-1})^{1-\delta}{\rm d}x-\int_{\Omega}(u^n)^{1-\delta}{\rm d}x\right ]&\leq -\int_{\Omega}\frac{u^n-u^{n-1}}{(u^n)^\delta}{\rm d}x.
\end{aligned}
\end{equation}
Therefore, gathering the estimates \eqref{esti4b} and \eqref{ineqconvexity}, we get 
\begin{eqnarray}
\frac{\Delta_t}{2}\displaystyle\sum_{n=1}^{N'}\int_{\Omega}\big{(}\frac{u^n-u^{n-1}}{\Delta_t}\big{)}^2{\rm d}x&+&\frac{1}{p}\left[\int_{\Omega}|\nabla u^{N'}|^p{\rm d}x-\int_{\Omega}|\nabla u_0|^p{\rm d}x\right]+\nonumber\\
 \frac{1}{1-\delta}\bigg{[}\int_{\Omega}(u_0)^{1-\delta}{\rm d}x&-&\int_{\Omega}(u^{N'})^{1-\delta}{\rm d}x\bigg{]}\leq |\Omega|\frac{T}{2}\|h\|^2_{L^\infty(Q_T)}\label{secondenergyeqter}.
\end{eqnarray}
The above expression together with $\int_{\Omega}(u^n)^{1-\delta}{\rm d}x\leq \max\left\{\int_{\Omega}(\bar{u})^{1-\delta}{\rm d}x,\,\int_{\Omega}(\underline{u})^{1-\delta}{\rm d}x\right\}$ yields 
\begin{equation}
\label{equicont}
\frac{\partial\tilde{u}_{\Delta_t}}{\partial t}\;\mbox{ is bounded in }\;L^2(Q_T)\;\mbox{uniformly in }\,\Delta_t,
\end{equation}
\begin{equation}
\label{sobolevbound}
u_{\Delta_t},\,\tilde{u}_{\Delta_t}\mbox{ are bounded in }L^\infty(0,T;\,W^{1,p}_0(\Omega))\;\mbox{ uniformly in }\Delta_t.
\end{equation}
Furthermore, from above there exists $C>0$ independent of $\Delta_t$ such that
\begin{eqnarray}\label{coincide}
\|u_{\Delta_t}-\tilde{u}_{\Delta_t}\|_{L^\infty(0,T;\, L^2(\Omega))}\leq\displaystyle\max_{n\in\{1,\cdot\cdot,N\}}\|u^n-u^{n-1}\|_{L^2(\Omega)}\leq C(\Delta_t)^{\frac{1}{2}}.
\end{eqnarray}
Therefore, taking $N\to \infty$ (which implies that $\Delta_t\to 0^+$), and up to a subsequence, we get from  \eqref{equicont} and \eqref{sobolevbound} that 
there exists $u, v\in L^\infty(0,T;\, W^{1,p}_0(\Omega)\cap L^\infty(\Omega))$ such that $\frac{\partial u}{\partial t}\in L^2(Q_T)$, $u,\,v\in \,{\mathcal C}$ uniformly and as $\Delta_t\to 0^+$,
\begin{equation}
\tilde{u}_{\Delta_t}\wstarconverge u\quad\mbox{in}\; L^\infty(0,T;\,W^{1,p}_0(\Omega)\cap L^\infty(\Omega)), \label{estimate1}
\end{equation}
\begin{equation}
u_{\Delta_t}\wstarconverge v\quad\mbox{in}\; L^\infty(0,T;\,W^{1,p}_0(\Omega)\cap L^\infty(\Omega)),\\ \label{estimate2}
\end{equation}
\begin{equation}
\frac{\partial\tilde{u}_{\Delta_t}}{\partial t}\rightharpoonup \frac{\partial u}{\partial t}\quad\mbox{in}\; L^2(Q_T).\label{estimate3}
\end{equation}
From \eqref{coincide}, it follows that $u\equiv\, v$.
Moreover, from \eqref{controluniform}, it follows that $\underline{u}\leq u \leq \bar{u}$. Therefore, $u\in {\bf V}(Q_T)$. 
\par 
Next, let us prove that $u$ satisfies (in the sense of Definition \ref{Def0}) the first equation in $(S_t)$. 
Using the boundedness of $\frac{\partial \tilde{u}_{\Delta_t}}{\partial t}$ in $ L^2(Q_T)$ given by \eqref{equicont}, we first get that $\{\Tilde u_{\Delta_t}\}_{\Delta_t}$ is equicontinuous in $C(0,T;\,L^q(\Omega))$ for $1\leq q\leq 2$, and thus with $\underline{u}\leq \tilde{u}_{\Delta_t} \leq \overline{u}$  and the interpolation inequality $\|\cdot \|_r\leq\|\cdot \|_\infty^\alpha\|\cdot \|_2^{1-\alpha}$, $\frac{1}{r}=\frac{\alpha}{\infty}+\frac{1-\alpha}{2}$, we obtain that $\{\Tilde u_{\Delta_t}\}_{\Delta_t}$ is equicontinuous in $C(0,T;\,L^q(\Omega))$ for any $1<q<+\infty$. Moreover, since $\{\Tilde u_{\Delta_t}\}_{\Delta_t}$ is a bounded family of $W_0^{1,p}(\Omega)$ which is compactely embedded in $L^q(\Omega)$ for $1<q<\frac{{\rm N}p}{{\rm N}-p}$, and from  Ascoli-Arzela Theorem, and using again the interpolation inequality, we get as $\Delta_t\to 0^+$ that up to a subsequence
\begin{equation}\label{compactness}
\tilde{u}_{\Delta_t}\rightarrow\, u\quad\mbox{in}\; C( 0,T;\,L^q(\Omega)),\;\forall\, q>1,
\end{equation}
and then, from \eqref{coincide} (with the interpolation inequality for $q>2)$, it follows that
\begin{equation}\label{compactnessb}
u_{\Delta_t}\rightarrow\, u\quad\mbox{in}\; L^\infty( 0,T;\,L^q(\Omega)),\;\forall\, q>1
\end{equation}
as $\Delta_t\to 0^+$. Thus, multiplying \eqref{distributions} by $(u_{\Delta_t}-u)$ and using \eqref{compactness}-\eqref{compactnessb}, we get by straightforward calculations:
\begin{eqnarray*}
\int_0^T\int_{\Omega}\left[\frac{\partial\tilde{u}_{\Delta_t}}{\partial t}-\frac{\partial u}{\partial t}\right](\tilde{u}_{\Delta_t}-u){\rm d}x{\rm d}t&-&\int_0^T<\Delta_p u_{\Delta_t}, u_{\Delta_t}-u>{\rm d}t\\
-\int_0^T\int_{\Omega}u_{\Delta_t}^{-\delta}(u_{\Delta_t}-u){\rm d}x{\rm d}t&=&\int_0^T\int_{\Omega}h_{\Delta_t}(u_{\Delta_t}-u){\rm d}x{\rm d}t+o_{\Delta_t}(1).
\end{eqnarray*}
From \eqref{controluniform} and  \eqref{compactness}, we have that 
\begin{eqnarray*}
\int_0^T\int_{\Omega}u_{\Delta_t}^{-\delta}(u_{\Delta_t}-u){\rm d}x{\rm d}t=o_{\Delta_t}(1),
\end{eqnarray*}
and from \eqref{eq2} and \eqref{compactness} we have
\begin{eqnarray*}
\int_0^T\int_{\Omega}h_{\Delta_t}(u_{\Delta_t}-u){\rm d}x{\rm d}t=o_{\Delta_t}(1).
\end{eqnarray*}
Then,
\begin{eqnarray*}
\frac{1}{2}\int_{\Omega}|\tilde{u}_{\Delta_t}-u|^2(T){\rm d}x-\int_0^T<\Delta_p u_{\Delta_t}-\Delta_p u,u_{\Delta_t}-u>{\rm d}t=o_{\Delta_t}(1).
\end{eqnarray*}
Therefore, using \eqref{compactnessb}, $u\not\equiv\, 0$ and the inequality \eqref{inegaliteAlg} with $w=u_{\Delta_t}(t)$ and $v=u(t)$
we obtain that $u_{\Delta_t}\rightarrow u$ in $L^p(0,T;\,W^{1,p}_0(\Omega))$ and
\begin{equation}
-\Delta_pu_{\Delta_t}\rightarrow -\Delta_p u\quad\mbox{in}\; L^{\frac{p}{p-1}}(0,T;\,W^{-1,\frac{p}{p-1}}(\Omega)). \label{eq5}
\end{equation}
Moreover, from \eqref{controluniform}, for any $\phi\in W^{1,p}_0(\Omega)$
\begin{eqnarray*}
\left|\int_{\Omega}\frac{\phi}{(u_{\Delta_t})^\delta}{\rm d}x\right|\leq\int_{\Omega}\frac{|\phi|}{(\underline{u})^\delta}{\rm d}x\leq \left(\int_{\Omega}\left(\frac{d(x)}{(\underline{u})^\delta}{\rm d}x\right)^{\frac{p}{p-1}}\right)^{\frac{p-1}{p}}\times\left (\int_{\Omega}\left(\frac{|\phi|}{d(x)}\right)^p{\rm d}x\right)^{\frac{1}{p}}
\end{eqnarray*}
and since $\delta<2+\frac{1}{p-1}$
\begin{eqnarray*}
\int_{\Omega}\left(\frac{d(x)}{\underline{u}^\delta}{\rm d}x\right)^{\frac{p}{p-1}}<+\infty.
\end{eqnarray*}
Then, from the Hardy Inequality and from the Lebesgue Theorem, we obtain
\begin{equation}
\frac{1}{{(u_{\Delta_t})}^\delta}\rightarrow\, \frac{1}{u^\delta}\quad\mbox{in}\; L^\infty(0,T;\,W^{-1,\frac{p}{p-1}}(\Omega)).\label{eq4}
\end{equation}
Therefore, from \eqref{eq2bis}, \eqref{compactness}, \eqref{compactnessb}, \eqref{eq5}, \eqref{eq4} we deduce that $u\in {\bf V}(Q_T)$ satisfies $({\rm P}_t)$.
\par Let us now show that $u$ is the unique weak solution  such that $u(t)\in {\mathcal C}$, $\forall \, t\in [0,T]$. Assume that there exists $v\not\equiv\, u$ a weak solution to $({\rm P}_t)$ satisfying $v(t)\in {\mathcal C}$, $\forall \, t\in [0,T]$. Then,
\begin{eqnarray*}
 \int_0^T\int_{\Omega}\frac{\partial (u-v)}{\partial t}(u-v){\rm d}x{\rm d}t&-&\int_0^T<\Delta_pu-\Delta_p v, u-v>{\rm d}t\\
&-&\int_0^T\int_{\Omega}(u^{-\delta}-v^{-\delta})(u-v){\rm d}x{\rm d}t=0.
\end{eqnarray*}
The above equality together with $u(0)=v(0)$ imply $u\equiv\, v$.
\par To complete the proof of Theorem~\ref{3}, let us prove $u\in C([0,T];W^{1,p}_0(\Omega))$ and \eqref{secondenergyeq}. First, we observe that since $u\in C([0,T],L^2(\Omega))$ and $u\in L^\infty(0,T;\, W^{1,p}_0(\Omega))$, it follows that $u\, :t\in[0,T]\to W^{1,p}_0(\Omega)$ is weakly continuous and then that $u(t_0)\in W^{1,p}_0(\Omega)$ and $\|u(t_0)\|_{W^{1,p}_0(\Omega)}\leq \displaystyle\liminf_{t\to t_0}\|u(t)\|_{W^{1,p}_0(\Omega)}$ for all $t_0\in [0,T]$. From \eqref{eq2bis}, \eqref{secondenergyeqter} (with $\sum_{n=N''}^{N'}$ for $1\leq N''\leq N' $ instead of $\sum_{n=1}^{N'}$), \eqref{coincide} and \eqref{eq4}, it follows that $u$ satisfies for any $t\in [t_0,T]$:
\begin{eqnarray}
\int_{t_0}^t\int_{\Omega}(\frac{\partial u}{\partial t})^2{\rm d}x{\rm d}s+\frac{1}{p}\int_{\Omega}|\nabla u(t)|^p{\rm d}x-\frac{1}{1-\delta}\int_{\Omega}u(t)^{1-\delta}{\rm d}x\leq \int_{t_0}^t\int_{\Omega}h\frac{\partial u}{\partial t}{\rm d}x{\rm d}s+\nonumber\\
\frac{1}{p}\int_{\Omega}|\nabla u(t_0)|^p{\rm d}x-\frac{1}{1-\delta}\int_{\Omega}u(t_0)^{1-\delta}{\rm d}x.\label{secinequality}
\end{eqnarray}
From \eqref{secinequality} and Lebesgue Theorem, it follows that $$\displaystyle\limsup_{t\to {t_0}^+}\|u(t)\|_{W^{1,p}_0(\Omega)}\leq \|u(t_0)\|_{W^{1,p}_0(\Omega)}$$ and then $u(t)\to u(t_0)$ in $W^{1,p}_0(\Omega)$ as $t\to {t_0}^+$ which implies that $u$ is right\--continuous on $[0,T]$. Let $t>t_0$. Let us now prove the left-continuity. Let $0< k\leq t-t_0$. Multiplying \eqref{linearheateq} by $\tau_k(u)(s)\eqdef \frac{u(s+k)-u(s)}{k}$ and integrating over $(t_0,t)\times \Omega$, we get, using convexity arguments, that
\begin{equation}
\begin{aligned}
&\int_{t_0}^t\int_{\Omega}\frac{\partial u}{\partial t}\tau_k(u){\rm d}x{\rm d}s+\frac{1}{kp}\int_{t_0}^t\int_{\Omega}(|\nabla u(s+k)|^p-|\nabla u(s)|^p){\rm d}x{\rm d}s \nonumber \\
& -\frac{1}{(1-\delta)k}\int_{t_0}^t\int_{\Omega}(u^{1-\delta}(s+k)-u^{1-\delta}(s)){\rm d}x{\rm d}s\geq 
\int_{t_0}^t\int_{\Omega}h\tau_k(u){\rm d}x{\rm d}s.
\end{aligned}
\end{equation}
Then, 
\begin{eqnarray*}
\int_{t_0}^t\int_{\Omega}\frac{\partial u}{\partial t}\tau_k(u){\rm d}x{\rm d}s+\frac{1}{kp}(\int_t^{t+k}\int_{\Omega}|\nabla u(s)|^p{\rm d}x{\rm d}s-\int_{t_0}^{t_0+k}\int_{\Omega}|\nabla u(s)|^p{\rm d}x{\rm d}s)\\
-\frac{1}{(1-\delta)k}(\int_t^{t+k}\int_{\Omega}(u^{1-\delta}(s){\rm d}x{\rm d}s-\int_{t_0}^{t_0+k} u^{1-\delta}(s)){\rm d}x{\rm d}s)\geq 
\int_{t_0}^{t}\int_{\Omega}h\tau_k(u){\rm d}x{\rm d}s.
\end{eqnarray*}
Since $u$ is right\--continuous in $W^{1,p}_0(\Omega)$ and by Lebesgue theorem, we get as $k\to 0^+$
\begin{eqnarray*}
\frac{1}{k}\int_t^{t+k}\int_{\Omega}|\nabla u(s)|^p{\rm d}x{\rm d}s&\to& \int_{\Omega}|\nabla u(t)|^p{\rm d}x,\\
\frac{1}{k}\int_{t_0}^{t_0+k}\int_{\Omega}|\nabla u(s)|^p{\rm d}x{\rm d}s&\to& \int_{\Omega}|\nabla u(t_0)|^p{\rm d}x,\\
\frac{1}{k}\int_t^{t+k}\int_{\Omega}u^{1-\delta}(s){\rm d}x{\rm d}s&\to&\int_{\Omega}u^{1-\delta}(t){\rm d}x,\\
\frac{1}{k}\int_{t_0}^{t_0+k}\int_{\Omega}u^{1-\delta}(s)){\rm d}x{\rm d}s&\to& \int_{\Omega}u(t_0)^{1-\delta}{\rm d}x.
\end{eqnarray*}
From the above estimates, we get as $k\to 0^+$
\begin{eqnarray*}
&\displaystyle \int_{t_0}^t\int_{\Omega}(\frac{\partial u}{\partial t})^2{\rm d}x{\rm d}s+\frac{1}{p}\int_{\Omega}|\nabla u(t)|^p{\rm d}x-\frac{1}{1-\delta}\int_{\Omega}u(t)^{1-\delta}{\rm d}x\geq \\
&\displaystyle \int_{t_0}^t\int_{\Omega}h\frac{\partial u}{\partial t}{\rm d}x{\rm d}s+\frac{1}{p}\int_{\Omega}|\nabla u(t_0)|^p{\rm d}x-\frac{1}{1-\delta}\int_{\Omega}u(t_0)^{1-\delta}{\rm d}x,
\end{eqnarray*}
which implies together with \eqref{secinequality} that the above inequality is in fact an equality. Then together with the fact that $t\to \int_{\Omega}u^{1-\delta}(t){\rm d}x$ is continuous, it follows that $u\in C([0,T],W^{1,p}_0(\Omega))$. Finally \eqref{secondenergyeq} is obtained by setting $t_0=0$.
\end{proof}
We end this section by proving proposition~\ref{regsobolev}.
\begin{proof}
Assume that $u_0\in \overline{{\mathcal D}(A)}^{L^\infty(\Omega)}$, where $A$ and ${\mathcal D}(A)$ are defined in \eqref{def-domaine-op}. From \eqref{monotonicity2}, $A$ is m-accretive in $L^\infty(\Omega)$. Indeed, for $f,g\in L^\infty(\Omega)$ and $\lambda>0$, set $u$ and $v\in {\mathcal D}(A)$ (given by Theorem~\ref{1}) satisfying
\begin{equation}
\label{accretive}
\begin{aligned}
& u-\lambda Au=f\quad\mbox{in }\,\Omega,\\
& v-\lambda Av=g\quad\mbox{in }\,\Omega.
\end{aligned}
\end{equation}
From \eqref{accretive} and defining $w\eqdef \left(u-v-\|f-g\|_{L^\infty(\Omega)}\right)^+$, we get that 
\begin{equation}
\nonumber
\int_{\Omega}w^2{\rm d}x+\lambda<Au-Av,w>_{W^{-1,p'}(\Omega),W^{1,p}_0(\Omega)}\leq 0.
\end{equation}
From \eqref{inegaliteAlg} or \eqref{inegaliteAlg2} it follows that $u-v\leq\|f-g\|_{L^\infty(\Omega)}$ and reversing the roles of $u$ and $v$, we get that $\|u-v\|_{L^\infty(\Omega)}\leq\|f-g\|_{L^\infty(\Omega)}$. Then Proposition~\ref{regsobolev}
can be obtained from \cite[Chap.4, Th. 4.2 and Th. 4.4]{Ba}. However, in order to be complete and self-contained, let us briefly explain the argument. In the following, $\|\cdot\|_\infty$ stands for the norm of $L^\infty(\Omega)$. For $z\in {\mathcal D}(A)$ and $r,k$ in $L^\infty(Q_T)$ let define 
$$
\varphi(t,s)=\|r(t)-k(s)\|_{\infty}\quad (t,s)\in [0,T]\times [0,T],
$$
and
\begin{eqnarray*}
b(t,r,k)&=&\|u_0-z\|_{\infty}+ \|v_0-z\|_{\infty}+|t|\|A z \|_{\infty}\\
&+&\int_0^{t^+}\|r(\tau)\|_{\infty} {\rm d}\tau+\int_0^{t^-}\|k(\tau)\|_{\infty} {\rm d}\tau,\;t\in [-T,T],
\end{eqnarray*}
and
$$
\Psi(t,s)=b(t-s,r,k)+
\left\{
\begin{array}{ll}
\displaystyle \int_0^{s}\varphi(t-s+\tau,\tau){\rm d}\tau & \mbox{ if } 0\leq s\leq t\leq T,\\
\displaystyle \int_0^{t}\varphi(\tau,s-t+\tau) {\rm d}\tau &  \mbox{ if } 0\leq t\leq s\leq T,
\end{array}
\right.
$$
the solution of
\begin{equation}\label{eqPsi}
\left\{
\begin{array}{rcll}
\displaystyle \frac{\partial \Psi}{\partial t}(t,s)+\frac{\partial \Psi}{\partial s}(t,s)&=&\displaystyle\varphi(t,s)\; &(t,s)\in [0,T]\times [0,T],\\
\displaystyle \Psi(t,0)&=& \displaystyle b(t,r,k)\; &t\in [0,T],\\
 \displaystyle  \Psi(0,s)& =& \displaystyle b(-s,r,k)\; &s\in [0,T].
\end{array}
\right.
\end{equation}
Moreover, let denote by $(u_\epsilon^n)$ the solution of \eqref{scheme} with $\Delta_t=\epsilon$, $h=r$, $r^n=\frac{1}{\epsilon}\int_{(n-1)\epsilon}^{n\epsilon}r(\tau,\cdot)d\tau$ and  $(u_\eta^n)$ the solution of \eqref{scheme} with  $\Delta_t=\eta$, $h=k$, $k^n=\frac{1}{\eta}\int_{(n-1)\eta}^{n\eta}k(\tau,\cdot)d\tau$ respectively. For $(n,m)\in \mathbb{N}^*$ elementary calculations leads to
\begin{eqnarray*}
u_\epsilon^n-u_\eta^m+\frac{\epsilon\eta}{\epsilon+\eta}(A u_\epsilon^n-A u_\eta^m)&=&\frac{\eta}{\epsilon+\eta}(u_\epsilon^{n-1}-u_\eta^m)\\
+\frac{\epsilon}{\epsilon+\eta}(u_\epsilon^n-u_\eta^{m-1})&+&\frac{\epsilon\eta}{\epsilon+\eta}(r^n-k^m),
\end{eqnarray*}
and since $A$ is m-accretive in $L^\infty(\Omega)$ we first verify that $\Phi_{n,m}^{\epsilon,\eta}=\|u_\epsilon^n-u_\eta^m\|_\infty$ obeys
\begin{eqnarray*}
\Phi_{n,m}^{\epsilon,\eta}&\leq& \frac{\eta}{\epsilon+\eta}\Phi_{n-1,m}^{\epsilon,\eta}+\frac{\epsilon}{\epsilon+\eta}\Phi_{n,m-1}^{\epsilon,\eta}+\frac{\epsilon\eta}{\epsilon+\eta}\|r^n-k^m\|_\infty,\\
\Phi_{n,0}^{\epsilon,\eta}&\leq & b(t_n,r_\epsilon,k_\eta)\quad\mbox{ and }\quad  \Phi_{0,m}^{\epsilon,\eta}\leq  b(-s_m ,r_\epsilon,k_\eta),
\end{eqnarray*}
and thus, with an easy inductive argument, that $\Phi_{n,m}^{\epsilon,\eta}\leq \Psi_{n,m}^{\epsilon,\eta}$ where $\Psi_{n,m}^{\epsilon,\eta}$ satisfies
\begin{eqnarray*}
\Psi_{n,m}^{\epsilon,\eta}&=& \frac{\eta}{\epsilon+\eta}\Psi_{n-1,m}^{\epsilon,\eta}+\frac{\epsilon}{\epsilon+\eta}\Psi_{n,m-1}^{\epsilon,\eta}+\frac{\epsilon\eta}{\epsilon+\eta}\|h_\epsilon^n-h_\eta^m\|_\infty,\\
\Psi_{n,0}^{\epsilon,\eta}&= & b(t_n,r_\epsilon,k_\eta)\quad\mbox{ and }\quad  \Psi_{0,m}^{\epsilon,\eta}=  b(-s_m ,r_\epsilon,k_\eta).
\end{eqnarray*}
For $(t,s)\in (t_{n-1},t_n)\times (s_{m-1},s_m)$ let set $\varphi^{\epsilon,\eta}(t,s)=\|r_\epsilon(t)-k_\eta(s)\|_{\infty}$ and $\Psi^{\epsilon,\eta}(t,s)=\Psi_{n,m}^{\epsilon,\eta}$,  $b_{\epsilon,\eta}(t,r,k)=b(t_n,r_\epsilon,k_\eta)$ and $b_{\epsilon,\eta}(-s,r,k)=b(-s_m,r_\epsilon,k_\eta)$. Then $\Psi^{\epsilon,\eta}$ satisfies the following discrete version of \eqref{eqPsi}:
\begin{eqnarray*}
\frac{\Psi^{\epsilon,\eta}(t,s)-\Psi^{\epsilon,\eta}(t-\epsilon,s)}{\epsilon}+\frac{\Psi^{\epsilon,\eta}(t,s)-\Psi^{\epsilon,\eta}(t,s-\eta)}{\eta}&=& \varphi^{\epsilon,\eta}(t,s),\\
\Psi^{\epsilon,\eta}(t,0)= b_{\epsilon,\eta}(t,r,k)\quad\mbox{ and }\quad  \Psi^{\epsilon,\eta}(0,s)&=&  b_{\epsilon,\eta}(s,r,k),
\end{eqnarray*}
and from $b_{\epsilon,\eta}(\cdot,r,k)\to b(\cdot,r,k)$ in $L^\infty([0,T])$ and $\varphi^{\epsilon,\eta}\to \varphi $ in $L^\infty([0,T]\times[0,T])$ we deduce that $\rho_{\epsilon,\eta}=\|\Psi^{\epsilon,\eta}-\Psi\|_{L^\infty([0,T]\times [0,T])} \to 0$ as $(\epsilon,\eta)\to 0$ (see for instance  \cite[Chap.4, Lemma~4.3]{Ba}). Then from
\begin{equation}
\|u_\epsilon(t)-u_\eta(s)\|_{\infty}=\Phi^{\epsilon, \eta}(t,s)\leq \Psi^{\epsilon, \eta}(t,s)\leq \Psi(t,s)+\rho_{\epsilon,\eta},\label{eqconv}
\end{equation}
we obtain with $t=s$, $r=k=h$, $v_0=u_0$:
$$
\|u_\epsilon(t)-u_\eta(t)\|_{\infty}\leq 2\|u_0-z\|_{\infty}+\rho_{\epsilon, \eta},
$$
and since $z$ can be chosen in $\mathcal{D}(A)$ arbitrary close to $u_0$, we deduce that $u_\epsilon$ is a Cauchy sequence in $L^\infty(Q_T)$ and then that $u_\epsilon \to u$ in $L^\infty(Q_T)$. Thus, passing to the limit in \eqref{eqconv} with $r=k=h$, $v_0=u_0$ we obtain 
\begin{eqnarray}
\|u(t)-u(s)\|_\infty &\leq& 2\|u_0-z\|_{\infty}+|t-s|\|A z \|_{\infty} +\int_0^{|t-s|}\|h(\tau)\|_{\infty}{\rm d}\tau \nonumber \\
&+&\int_{0}^{\max(t,s)}\|h(|t-s|+\tau)-h(\tau)\|_{\infty}{\rm d}\tau,\label{uCont}
\end{eqnarray}
which, together with the density ${\mathcal D}(A)$ in $L^\infty(\Omega)$ and  $h\in L^1(0,T;L^\infty(\Omega))$, yields $u\in C([0,T];L^\infty(\Omega))$. Analogously, from \eqref{eqconv} with $\varepsilon=\eta=\Delta_t$, $r=k=h$, $v_0=u_0$ and $t=s+\Delta_t$ we deduce that
\begin{eqnarray*}
\|u_{\Delta_t}(t)-\Tilde u_{\Delta_t}(t)\|_{\infty}&\leq& 2\| u_{\Delta_t}(t)- u_{\Delta_t}(t-\Delta_t)\|_{\infty}\\
& \leq& 4\|u_0-z\|_{\infty}+2\Delta_t\|A z \|_{\infty} +2\int_0^{\Delta_t } \|h(\tau)\|_{\infty}{\rm d}\tau\\
&+&2\int_{0}^{t}\|h(\Delta_t+\tau)-h(\tau)\|_{\infty}{\rm d}\tau,
\end{eqnarray*}
which gives the limit $\Tilde u_{\Delta_t}\to u$ in $C([0,T];L^\infty(\Omega))$ as $\Delta_t\to 0^+$. Note that since $\Tilde u_{\Delta_t}\in C([0,T];C_0(\overline{\Omega})$, the uniform limit $u$ belongs to $C([0,T];C_0(\overline{\Omega}))$. 
Moreover, passing to the limit in \eqref{eqconv} with $t=s$ we obtain
$$
\|u(t)-v(t)\|_{\infty}\leq \|u_0-z\|_{\infty}+\|v_0-z\|_{\infty}+\int_0^t\|r(\tau)-k(\tau)\|_{\infty} {\rm d}\tau,
$$
and \eqref{estimation} follows because we can choose $z$ arbitrary close to $v_0$. 
Finally, if $A u_0\in L^\infty(\Omega)$ and $h\in W^{1,1}(0,T;L^\infty(\Omega))$ and if we assume (without loss of generality) that $t>s$ then with $z=v_0=u(t-s)$ and $(r,k)=(h,h(\cdot+t-s))$ in the last above inequality we obtain 
\begin{eqnarray}
\|u(t)-u(s)\|_{\infty}&\leq& \|u_0-u(t-s)\|_{\infty}+\int_{0}^{s}\|h(\tau)-h(\tau+t-s)\|_{\infty} {\rm d}\tau\nonumber\\
&\leq &\int_0^{t-s}\|A u_0-h(\tau)\|_{\infty}{\rm d}\tau+\int_{0}^{s}\|h(\tau)-h(\tau+t-s)\|_{\infty} {\rm d}\tau\nonumber\\
&\leq & (t-s)\|A u_0-h(0)\|_{\infty}+\int_{0}^{t-s}\|h(0)-h(\tau)\|_{\infty} {\rm d}\tau\nonumber\\
&+& \int_{0}^{s}\|h(\tau)-h(\tau+t-s)\|_{\infty} {\rm d}\tau\nonumber\\
&\leq & (t-s)\left((\|A u_0-h(0)\|_{\infty}+\int_0^{T}\left \|\frac{{\rm d}h(\tau)}{{\rm d}t}\right\|_{\infty} {\rm d}\tau\right)\label{mainestimate}.
\end{eqnarray}
Note that the second above inequality is obtained from  \eqref{estimation} with $v=u_0$, $k=A u_0$ and the last above inequality is obtained from $h(\tau)-h(\tau+t-s)=\int_\tau^{\tau+t-s}\frac{{\rm d}h(\tau)}{{\rm d}t}{\rm d}\sigma$ together with Fubini's Theorem. Dividing the expression \eqref{mainestimate} by $|t-s|$, we get that $u$ is a Lipschitz function  and since $\frac{\partial u}{\partial t}\in L^2(Q_T)$, passing to the limit $|t-s|\to 0$ we obtain that $\frac{u(t)-u(s)}{t-s}\to \frac{\partial u}{\partial t}$ as $s\to t$  weakly in $L^2(Q_T)$ and *-weakly in $L^\infty(Q_T)$. Furthermore,
\begin{equation}
\nonumber
\left\|\frac{\partial u}{\partial t}\right\|_{\infty}\leq \displaystyle\liminf_{s\to t}\frac{\|u(t)-u(s)\|_{\infty}}{|t-s|}.
\end{equation}
Therefore, we get $u\in W^{1,\infty}(0,T;L^\infty(\Omega))$ as well as inequality \eqref{estimationbiss}. 
\end{proof}
\section{Proofs of Theorems~\ref{2bis}, \ref{3bis}, \ref{3ter} and of Proposition~\ref{regsobolevbis}}
\label{nonlinearcase}
We start this section by the proof of Theorem~\ref{2bis}. We use the following preliminary result which gives the validity of the weak comparison principle for sub\--homogeneous problems and then forces uniqueness of solutions.
\begin{theorem}
\label{WCP}
Let $1<r<+\infty$, $g:\Omega\times\R^+\rightarrow \R$ be a Caratheodory function bounded  below such that  $\frac{g(x,s)}{s^{r-1}}$ is a decreasing function in $\R^+$ for a.e. $x\in\Omega$. 
Let $u\in L^{\infty}(\Omega)\cap W^{1,r}_0(\Omega)$, $v\in L^\infty(\Omega)\cap W^{1,r}_0(\Omega)$ satisfy $u>0$, $v>0$ in $\Omega$, $\int_{\Omega}{u}^{1-\delta}\,{\rm d}x<+\infty$, $\int_{\Omega}{v}^{1-\delta}\,{\rm d}x<+\infty$ and
\begin{equation}
\nonumber
-\Delta_r u\leq \frac{1}{u^\delta}+g(x,u)\;\mbox{weakly in}\;W^{-1,\frac{r}{r-1}}(\Omega)
\end{equation}
and
\begin{equation}
\nonumber
-\Delta_rv\geq \frac{1}{v^\delta}+g(x,v)\;\mbox{weakly in}\; W^{-1,\frac{r}{r-1}}(\Omega).
\end{equation}
Suppose in addition that there exists  a positive function $u_0\in L^\infty(\Omega)$ and positive constants $c$, $C$ such that $cu_0\leq u,v\leq Cu_0$ and
\begin{equation}
\label{sublinearity+}
\int_\Omega \vert g(x,c u_0)\vert u_0{\rm d}x<+\infty, \quad \int_\Omega \vert g(x,C u_0)\vert u_0{\rm d}x<+\infty.
\end{equation}
Then, $u\leq v$.
\end{theorem}
\begin{proof}
First, for $\epsilon>0$, we set $u_\epsilon\eqdef u+\epsilon$ and $v_\epsilon\eqdef v+\epsilon$. Following some ideas in {\sc Lindqvist} \cite{Linq} (see also {\sc Dr\'abek-Hern\'andez} \cite{DrHe}) we use the Di\'az-Saa inequality (see {\sc di\'az-Saa} \cite{DiSa}) in the following way:\\
Let 
\begin{equation}
\nonumber
\displaystyle\phi\eqdef\frac{u_\epsilon^{r}-v_\epsilon^{r}}{u_\epsilon^{r-1}}\;\mbox{and}\;\psi\eqdef\frac{v_\epsilon^{r}-u_\epsilon^{r}}{v_\epsilon^{r-1}}.
\end{equation}
Setting $\Omega^+=\left\{x\in\Omega|u(x)>v(x)\right\}$, we have that $\phi\geq 0$ and $\psi\leq 0$ in $\Omega^+$ and
\begin{equation}
\nonumber
\int_{\Omega^+}|\nabla u|^{r-2}\nabla u\nabla\phi\,{\rm d}x\leq\int_{\Omega^+}\frac{\phi}{u^\delta}\,{\rm d}x+\int_{\Omega^+}g(x,u)\phi\,{\rm d}x<+\infty,
\end{equation}
\begin{equation}
\nonumber
\int_{\Omega^+}|\nabla v|^{r-2}\nabla v\nabla\psi\,{\rm d}x\leq\int_{\Omega^+}\frac{\psi}{v^\delta}\,{\rm d}x+\int_{\Omega^+}g(x,v)\psi\,{\rm d}x<+\infty.
\end{equation}
Since 
\begin{equation}
\nonumber
\displaystyle\nabla\phi=\left[1+(r-1)\left(\frac{u_\epsilon}{v_\epsilon}\right)^r\right]\nabla u -r\left(\frac{v_\epsilon}{u_\epsilon}\right)^{r-1}\nabla v
\end{equation}
and
\begin{equation}
\nonumber
\displaystyle\nabla\psi=\left[1+(r-1)\left(\frac{v_\epsilon}{u_\epsilon}\right)^r\right]\nabla v -r\left(\frac{u_\epsilon}{v_\epsilon}\right)^{r-1}\nabla u,
\end{equation}
we have 
\begin{equation}
\nonumber
\begin{aligned}
& \displaystyle\int_{\Omega^+}|\nabla u|^{r-2}\nabla u\nabla\phi\,{\rm d}x
+\displaystyle\int_{\Omega^+}|\nabla v|^{r-2}\nabla v\nabla\psi\,{\rm d}x=\\
& \displaystyle\int_{\Omega^+}\left(|\nabla u_\epsilon|^r\left[1+(r-1)\left(\frac{v_\epsilon}{u_\epsilon}\right)^r\right]
+|\nabla v_\epsilon|^r\left[1+(r-1)\left(\frac{u_\epsilon}{v_\epsilon}\right)^r\right]\right )\,{\rm d}x\\
& =\displaystyle\int_{\Omega^+}(u_\epsilon^r-v_\epsilon^r)(|\nabla \log{u_\epsilon}|^r-|\nabla \log{v_\epsilon}|^r)\,{\rm d}x\\
& -\displaystyle\int_{\Omega^+}rv_\epsilon^r|\nabla \log{u_\epsilon}|^{r-2}\nabla \log{u_\epsilon}(\nabla \log{v_\epsilon}-\nabla \log{u_\epsilon})\,{\rm d}x\\
&  -\displaystyle\int_{\Omega^+}ru_\epsilon^r|\nabla \log{v_\epsilon}|^{r-2}\nabla \log{v_\epsilon}(\nabla \log{u_\epsilon}-\nabla \log{v_\epsilon})\,{\rm d}x.
\end{aligned}
\end{equation}
If $r\geq 2$, then using the following well-known inequality
\begin{equation}
\nonumber
|w_2|^r\geq |w_1|^r+r|w_1|^{r-2}w_1(w_2-w_1)+\frac{|w_2-w_1|^r}{2^{r-1}-1},
\end{equation}
for all points $w_1$ and $ w_2\in\R^{\rm N}$, we get that 
\begin{equation}
\nonumber
\begin{aligned}
& \displaystyle\int_{\Omega^+}|\nabla u|^{r-2}\nabla u\nabla\phi\,{\rm d}x+\displaystyle\int_{\Omega^+}|\nabla v|^{r-2}\nabla v\nabla\psi\,{\rm d}x\\
& \geq \displaystyle\int_{\Omega^+}(u_\epsilon^r-v_\epsilon^r)(|\nabla \log{u_\epsilon}|^r-|\nabla \log{v_\epsilon}|^r)\,{\rm d}x\\
& +\displaystyle\int_{\Omega^+}\left (v_\epsilon^r(|\nabla\log{u_\epsilon}|^r-|\nabla\log{v_\epsilon}|^r)+\frac{|\nabla\log{v_\epsilon}-\nabla\log{u_\epsilon}|^r}{2^{r-1}-1}v_\epsilon^r\right)\,{\rm d}x\\
& +\displaystyle\int_{\Omega^+}\left(u_\epsilon^r(|\nabla\log{v_\epsilon}|^r-|\nabla\log{u_\epsilon}|^r)+\frac{|\nabla\log{u_\epsilon}-\nabla\log{v_\epsilon}|^r}{2^{r-1}-1}u_\epsilon^r\right)\,{\rm d}x\\
& \geq \displaystyle\frac{1}{2^{r-1}-1}\int_{\Omega^+}|\nabla\log{v_\epsilon}-\nabla\log{u_\epsilon}|^r(v_\epsilon^r+u_\epsilon^r)\,{\rm d}x\\
& = \displaystyle\frac{1}{2^{r-1}-1}\int_{\Omega^+}|u_\epsilon\nabla v_\epsilon-v_\epsilon\nabla u_\epsilon|^r(\frac{1}{u_\epsilon^r}+\frac{1}{v_\epsilon^r})\,{\rm d}x.
\end{aligned}
\end{equation}
If $1<r<2$ then using the following inequality (with some suitable $C(r)>0$)
\begin{equation}
\nonumber
|w_2|^r\geq |w_1|^r+r|w_1|^{r-2}w_1(w_2-w_1)+C(r)\frac{|w_2-w_1|^2}{(|w_1|+|w_2|)^{2-r}},
\end{equation}
for all points $w_1$ and $w_2\in\R^{\rm N}$, the last term of the above inequality dealing with $r\geq 2$ is replaced by the following term
\begin{equation}
\nonumber
\begin{aligned}
& \displaystyle C(r)\int_{\Omega^+}\frac{|\nabla\log{v_\epsilon}-\nabla\log{u_\epsilon}|^2}{(|\nabla\log{v_\epsilon}|+|\nabla\log{u_\epsilon}|)^{2-r}}(v_\epsilon^r+u_\epsilon^r)\,{\rm d}x\\
& =\displaystyle C(r)\int_{\Omega^+}(\frac{1}{u_\epsilon^r}+\frac{1}{v_\epsilon^r})\frac{|u_\epsilon\nabla v_\epsilon-v_\epsilon\nabla u_\epsilon|^2}{(u_\epsilon|\nabla v_\epsilon|+v_\epsilon|\nabla u_\epsilon|)^{2-r}}\,{\rm d}x.
\end{aligned}
\end{equation}
In the right hand side, we get:
\begin{equation}
\nonumber
\begin{aligned}
& \int_{\Omega^+}\left(\frac{1}{u^\delta}+g(x,u)\right)\phi\,{\rm d}x+\int_{\Omega^+}\left(\frac{1}{v^\delta}+g(x,v)\right)\psi\,{\rm d}x\leq\\
& \int_{\Omega^+}\left[\frac{g(x,u)}{u^{r-1}}\left(\frac{u}{u_\epsilon}\right)^{r-1}-\frac{g(x,v)}{v^{r-1}}\left(\frac{v}{v_\epsilon}\right)^{r-1}\right](u_\epsilon^r-v_\epsilon^r)\,{\rm d}x.
\end{aligned}
\end{equation}
 Then, since $\frac{u}{u_\epsilon}\to 1$, $\frac{v}{v_\epsilon}\to 1$ as $\epsilon\to 0^+$ a.e. in $\Omega$, we get from \eqref{sublinearity+} and the Lebesgue Theorem
\begin{equation}
\nonumber
\displaystyle\lim_{\epsilon\to 0^+}\int_{\Omega^+}(g(x,u)\phi+g(x,v)\psi)\,{\rm d}x\leq 0.
\end{equation}
By Fatou lemma and using the above estimates, we obtain that $|u\nabla v-v\nabla u|=0$ a.e. in $\Omega^+$ from which we get that on each connected component set ${\mathcal O}$ of $\Omega^+$, there exists $k\in\R$ such that $u=kv$ in ${\mathcal O}$. From
\begin{equation}
\nonumber
\begin{aligned}
& \int_{{\mathcal O}}|\nabla u|^r\,{\rm d}x=k^r\int_{{\mathcal O}}|\nabla v|^r\,{\rm d}x\leq \int_{{\mathcal O}}(k^{1-\delta}v^{1-\delta}+g(x,kv)kv)\,{\rm d}x,\\
& k^r\int_{{\mathcal O}}|\nabla v|^r\,{\rm d}x\geq k^r\int_{{\mathcal O}}(v^{1-\delta}+g(x,v)v)\,{\rm d}x
\geq\int_{{\mathcal O}}(k^rv^{1-\delta}+g(x,kv)kv)\,{\rm d}x
\end{aligned}
\end{equation}
we get $k\leq 1$ which implies that $u\leq v$ in $\Omega^+$ and from the definition of $\Omega^+$, $u\leq v$ in $\Omega$.
\end{proof}
We now prove Theorem~\ref{2bis}.
\begin{proof}
For $0<\alpha_f<\ell<\lambda_1(\Omega)$, $A>0$ large enough and $0<\eta<M$ let define:
\begin{eqnarray}
\quad \phi =\left\{\begin{array}{ll}
 V\quad&\mbox{if}\; \delta<1,\\
\phi_1\left(\ln(\frac{A}{\phi_1})\right)^{\frac{1}{p}}\quad&\mbox{if}\; \delta=1,\\
\phi_1^{\frac{p}{p-1+\delta}}\quad&\mbox{if}\;\delta>1,
\end{array}\right.\mbox{ and }\quad 
\displaystyle \underline{u}=\eta \phi,
\label{subsol}
\end{eqnarray}
with $V$  a positive solution of 
$$
\left\{\begin{array}{rll}
\displaystyle -\Delta_p u&\displaystyle=\ell u^{p-1}+\frac{1}{u^\delta}\;&\displaystyle\mbox{in}\;\Omega,\\
\displaystyle u&\displaystyle=0\; &\displaystyle\mbox{on }\;\partial\Omega.
\end{array}
\right.
$$
The existence of $V$ follows from similar minimization and cut-off arguments given in {\sc Giacomoni-Schindler-Tak\'a\v{c} \cite[proof of Lemma 3.3, p.~126]{GiScTa}}. Note that from $\ell<\lambda_1(\Omega)$, the associated energy functional is coercive and weakly lower semicontinuous in $W^{1,p}_0(\Omega)$ and from Lemma~A.2 in \cite{GiScTa} G\^ateaux-differentiable in $W^{1,p}_0(\Omega)$. From Lemma~A.6 and Theorem~B.1 in \cite{GiScTa} and since $\delta<1$, $V\in C^{1,\alpha}(\overline{\Omega})$ for some $0<\alpha<1$.
\begin{eqnarray}
\label{supsol}
\displaystyle \overline{u}=\left\{\begin{array}{lc}
MV\quad{if}\;\delta<1\\
M\phi_1(\ln(\frac{A}{\phi_1})^{\frac{1}{p}})\quad{if}\;\delta=1\\
M(\phi_1+\phi) \quad{if}\;\delta>1.
\end{array}\right.
\end{eqnarray}
We verify that $\underline{u}\in \mathcal{C}$, $\overline{u}\in \mathcal{C}$. Let $L>0$ such that $-L\leq f(x,s)\leq \ell s^{p-1}+L$. We verify  that for $M>0$ large enough for $\eta>0$ small enough we have
$$
-\Delta_p \underline{u}-\frac{1}{\underline{u}^\delta}\leq -L \quad\mbox{in}\;\Omega,\quad \underline{u}=0\quad \mbox{ on }\partial\Omega,
$$
and 
$$
-\Delta_p\overline{u}-\frac{1}{\overline{u}^\delta}\geq \ell \overline{u}^{p-1}+L\quad\mbox{in}\;\Omega,\quad \overline{u}=0\quad \mbox{ on }\partial\Omega.
$$
\par We distinguish between the following two cases: the case where $\delta<1$ and the case where $\delta\geq 1$.
In the first case, the solution $u\in W^{1,p}_0(\Omega)$ to $({\rm Q})$ can be obtained as a global minimizer in $W^{1,p}_0(\Omega)$ of the functional $E$ defined below at $v\in W^{1,p}_0(\Omega)$:
$$
E(v)\eqdef\frac{1}{p}\int_{\Omega}|\nabla v|^p\,{\rm dx}-\int_{\Omega}G(x,v)\,{\rm dx}-\int_{\Omega}K(x,v)\,{\rm d}x,
$$
where for any $t\geq 0$
\begin{eqnarray*}
G(x,t)\eqdef\int_0^tg(x,s)\,{\rm d}s,\quad K(x,t)\eqdef\int_0^tk(x,s)\,{\rm d}s,
\end{eqnarray*}
and $g,k$ are the cut-off functions defined by
\begin{eqnarray*}
g(x,v(x))\eqdef\left\{\begin{array}{lc}
& v(x)^{-\delta}\;\mbox{if}\; v(x)\geq \underline{u}(x),\\
& \underline{u}(x)^{-\delta}\;\mbox{ otherwise},
\end{array}\right.
\end{eqnarray*}
and
\begin{eqnarray*}
k(x,v(x))\eqdef\left\{\begin{array}{ll}
\displaystyle  f(x,\underline{u}(x))\;&\displaystyle\mbox{if}\; v(x)\leq \underline{u}(x),\\
\displaystyle f(x,v(x))\;&\displaystyle\mbox{if}\; \underline{u}\leq v(x)\leq \overline{u}(x),\\
\displaystyle f(x,\overline{u}(x))\;&\displaystyle\mbox{if}\; v(x)\geq \overline{u}(x) .
\end{array}\right.
\end{eqnarray*}
Notice that the method of proof of Theorem \ref{1} when $\delta<1$ does not apply here because $\int_{\Omega}F(x,v) dx$ with $F(x,t)=\int_0^tf(x,s) ds$ is not convex in $v$. That is the reason why we introduce the above cut-off function. Since $f$ satisfies \eqref{sublineargrowth}, $E$ is coercive and  weakly lower semicontinuous in $W^{1,p}_0(\Omega)$. Using the compactness of any minimizing sequence $\{u_n\}$ in $L^p(\Omega)$ and the Lebesgue theorem, we can prove the existence of a global minimizer $u$ to $E$. From Lemma~A.2 in \cite{GiScTa}, we have that $E$ is G\^ateaux\--differentiable in $u$ and then $u$ satisfies:
$$
\left\{\begin{array}{rll}
\displaystyle -\Delta_p u-g(x,u)&\displaystyle=k(x,u) \;&\displaystyle\mbox{in }\,\Omega,\\
\displaystyle u&\displaystyle=0\; &\displaystyle\mbox{on }\;\partial\Omega.
\end{array}\right.
$$
Thus, from the weak comparison principle, we first get that $\underline{u}\leq u$ and then that $g(x,u)=u^{-\delta}$. Finally, still from the weak comparison principle we also obtain $u\leq \overline{u}$ from which we get $k(x,u)=f(x,u)$ and $u\in \mathcal{C}$.\par
Now, we deal with the second case. We use the following iterative scheme
\begin{equation}
\nonumber
\begin{array}{rll}
\displaystyle -\Delta_pu_n-\frac{1}{u_n^\delta}+Ku_n&\displaystyle=f(x,u_{n-1})+Ku_{n-1}\;&\displaystyle\mbox{in}\;\Omega,\\
\displaystyle u_n&\displaystyle=0\; &\displaystyle\mbox{on }\;\partial\Omega,
\end{array}
\end{equation}
with $u_0\eqdef\underline{u}$ and $K>0$ large enough such that $t\to Kt +f(x,t)$ is non decreasing (thanks to the uniform local lipschitz property of $f$) in $[0,\|\overline{u}\|_{L^\infty(\Omega)}]$ for a.e. $x\in\Omega$. Note that the iterative scheme is well-defined and produces a sequence of element $u_n\in W^{1,p}_0(\Omega)\cap \mathcal{C}\cap C_0(\overline{\Omega})$. From the weak comparison principle, we have that $(u_n)_{n\in\N}$ is a monotone increasing sequence
such that $u_n\leq\overline{u}$. 
Then $u_n\uparrow u$ in $C_0(\overline{\Omega})\cap \mathcal{C}$ and using the equation satisfied by $u_n$ we deduce that $(u_n)_{n\in\N}$ is a Cauchy sequence in $ W^{1,p}_0(\Omega)$ and then converges to $u $ in $W^{1,p}_0(\Omega)$. Thus, by passing to the limit in the equation satisfied by $u_n$ we obtain that $u$ is a solution to $({\rm Q})$. Finally, the uniqueness of $u$ follows from Theorem~\ref{WCP}.
\end{proof}
We now give the proof of Theorem~\ref{3bis}.
\begin{proof}
Let $T>0$, $N\in\N^*$ and $\Delta_t\eqdef \frac{T}{N}$. 
Following the main steps of the previous section, we are interested in constructing the sequence $(u^n)_{n\in\N}\subset L^\infty(\Omega)\cap W^{1,p}_0(\Omega)$, solutions to
\begin{equation}
\label{iterativeequation}
u^n-\Delta_t\left(\Delta_pu^n+\frac{1}{(u^n)^\delta}\right)=\Delta_t f(x,u^{n-1})+u^{n-1}\quad\mbox{in}\;\Omega.
\end{equation}
Applying Theorem~\ref{1} for each iteration $n$ and since $u_0\in {\mathcal C}\cap W^{1,p}_0(\Omega)$, we get the existence of $(u_n)_{n\in\N}\subset{\mathcal C}\cap W^{1,p}_0(\Omega)$. In fact, the previous inclusion is uniform in $\Delta_t$. Indeed, since $u_0\in {\mathcal C}$ we can choose $\underline{u}$ and $\overline{u}$ defined by \eqref{subsol} and \eqref{supsol} with $\eta>0$ and $M>0$ large enough so that $\underline{u}\leq u_0\leq \overline{u}$. Then with $f\geq -L$ the weak comparison principle guarantees $\underline{u}\leq u_n\leq \overline{u}$, with $\underline{u}$ and  $\overline{u}$ independent on $\Delta_t$. 
\par Next, let $u_{\Delta_t}$ and $ \tilde{u}_{\Delta_t}$ defined by \eqref{defUdiscret} and set $u_{\Delta_t}(t)=u_0$ if $t<0$. Then \eqref{distributions} is satisfied with $h_{\Delta_t}(t)\eqdef f(x,u_{\Delta_t}(t-\Delta_t))$ on $[0,T]$.
Notice that from $u_{\Delta_t}\in [\underline{u},\overline{u}]$ and $t\to f(x,t)$ continuous on  $[\underline{u},\overline{u}]$ it follows that $h_{\Delta_t}(t)$ is bounded in $L^\infty(Q_T)$ independently on $n$. Then by similar energy estimates as in the proof of Theorem~\ref{3} we get that
\begin{equation}
\label{estimateagain}
\begin{aligned}
& u_{\Delta_t},\; \tilde{u}_{\Delta_t}\in L^\infty(0,T;\,W^{1,p}_0(\Omega)\cap{\mathcal C}),\\
& \frac{\partial\tilde{u}_{\Delta_t}}{\partial t}\in L^2(Q_T),\\
& u_{\Delta_t},\; \tilde{u}_{\Delta_t}\in L^\infty(Q_T),\\
& \frac{1}{(u_{\Delta_t})^\delta}\in L^\infty(0,T;\,W^{-1,p'}(\Omega)),\\
& \|\tilde{u}_{\Delta_t}-u_{\Delta_t}\|_{L^2(\Omega)}\leq C(\Delta_t)^{\frac{1}{2}},
\end{aligned}
\end{equation}
uniformly on $\Delta_t$. Then, taking $\Delta_t\to 0$, it follows that there exists $u\in L^\infty(0,T;\, W^{1,p}_0(\Omega)$ such that $u\in L^\infty(Q_T)$ and, up to a subsequence, we have
\begin{equation}
\label{compactnessagain}
\begin{aligned}
& u_{\Delta_t},\, \tilde{u}_{\Delta_t}\wstarconverge u\quad\mbox{in}\; L^\infty(0,T;\,W^{1,p}_0(\Omega))\;\mbox{ and in}\; L^\infty(Q_T),\\
& \frac{\partial \tilde{u}_{\Delta_t}}{\partial t}\rightharpoonup\frac{\partial u}{\partial t}\quad\mbox{in}\; L^2(Q_T).
\end{aligned}
\end{equation}
Using similar arguments as in the proof of Theorem~\ref{3}, we get for $1<q<+\infty$,
\begin{equation}
\label{compactnessfinally}
u_{\Delta_t}, \tilde{u}_{\Delta_t}\to u\;\;\mbox{in}\; L^\infty(0,T;\,L^q(\Omega))\;\mbox{ and }\; u\in C([0,T],L^q(\Omega)).
\end{equation}
Moreover, if $K>0$ denotes the Lipschitz constant of $f$ on $[\underline{u},\overline{u}]$ we have
$$
\|f(x,u_{\Delta_t}(t-\Delta_t))-f(x,u(t))\|_{L^2(\Omega)}\leq K\|u_{\Delta_t}(t-\Delta_t)-u(t)\|_{L^2(\Omega)},
$$
and from \eqref{compactnessfinally} we deduce that $h_{\Delta_t}=f(x,u_{\Delta_t}(\cdot-\Delta_t))\to f(x,u)$ in $L^\infty(0,T;\, L^2(\Omega))$. Then by following the steps at the end of the proof of Theorem~\ref{3} we obtain that $u$ is a weak  solution to $({\rm P}_t)$ in ${\bf V}(Q_T)$. 
\par Next, let us prove that such a solution is unique. For that, let $v$ be a weak solution to  $({\rm P}_t)$ in ${\bf V}(Q_T)$. Since $f(x,\cdot)$ is locally lipschitz uniformly in $\Omega$, it follows that
\begin{eqnarray*}
\frac{1}{2}\|u-v\|_{L^2(\Omega)}^2&-&\int_0^T<\Delta_p u-\Delta_p v,u-v>{\rm d}t\\
-\int_0^T\int_{\Omega}\left(\frac{1}{u^\delta}-\frac{1}{v^\delta}\right)(u-v){\rm d}x{\rm d}t&=&\int_0^T\int_{\Omega}(f(x,u)-f(x,v)){\rm d}x{\rm d}t\\
&\leq&  C\int_0^T\int_{\Omega}|u-v|^2{\rm d}x{\rm d}t
\end{eqnarray*}
which implies together with the Gronwall Lemma and \eqref{monotonicity2}, that $u\equiv v$. Finally, as in the proof of Theorem~\ref{3}, we can prove that $u\in C([0,T]; W^{1,p}_0(\Omega))$ and that $u$ satisfies \eqref{secondenergyeqbis}.
\end{proof}
We now give the proof of Proposition~\ref{regsobolevbis}.
\begin{proof}
(i) is the consequence of \eqref{estimation} together with the fact that $f$ is locally lipschitz and the Gronwall Lemma.

Regarding assertion (ii), we follow the proof of Proposition \ref{regsobolev}:
Assume without loss of generality that $t>s$. Then, 
\begin{equation}
\nonumber
\|u(t)-u(s)\|_{L^\infty(\Omega)}\leq \|u_0-u(t-s)\|_{L^\infty(\Omega)}+\int_0^s\|f(x,u(\tau))-f(x,u(\tau+t-s))\|_{L^\infty(\Omega)}{\rm d}\tau.
\end{equation}
From assertion (i) and the fact that $f$ is Lipschitz on $[\underline{u},\overline{u}]$, it follows that
\begin{equation}
\nonumber
\begin{aligned}
 \|u(t)-u(s)\|_{L^\infty(\Omega)}&\leq \|u_0-u(t-s)\|_{L^\infty(\Omega)}+\omega\int_0^se^{\omega \tau}{\rm d}\tau \|u_0-u(t-s)\|_{L^\infty(\Omega)}\\
& \leq e^{\omega s}\|u_0-u(t-s)\|_{L^\infty(\Omega)}.
\end{aligned}
\end{equation}
Now, we estimate the term $\|u_0-u(t-s)\|_{L^\infty(\Omega)}$ in the following way:
\begin{equation}
\nonumber
\begin{aligned}
\|u_0-u(t-s)\|_{L^\infty(\Omega)}&\leq\int_0^{t-s}\|A u_0-f(x,u(\tau))\|_{L^\infty(\Omega)}{\rm d}\tau\\
&\leq (t-s)\|A u_0-f(x,u_0)\|_{L^\infty(\Omega)}+
\omega\int_0^{t-s}\|u_0-u(\tau)\|_{L^\infty(\Omega)}
{\rm d}\tau.
\end{aligned}
\end{equation}
From Gronwall lemma, we deduce that
\begin{equation}
\nonumber
\|u_0-u(t-s)\|_{L^\infty(\Omega)}\leq(t-s)e^{\omega(t-s)}\|A u_0-f(x,u_0)\|_{L^\infty(\Omega)}
\end{equation}
Gathering the above estimates, we get
\begin{equation}
\nonumber
\|u(t)-u(s)\|_{L^\infty(\Omega)}\leq(t-s)e^{\omega t}\|A u_0-f(x,u_0)\|_{L^\infty(\Omega)}.
\end{equation}
Then, the rest of the proof follows with the same arguments as in the proof of Proposition \ref{regsobolev}.
\end{proof}
To end this section, we prove Theorem~\ref{3ter}.
\begin{proof}
Let $\underline{u}$, $\overline{u}\in{\mathcal C}\cap W^{1,p}_0(\Omega)\cap C(\overline{\Omega})$ be the subsolution and supersolution to
\begin{equation}
\left\{\begin{array}{rll}
\displaystyle -\Delta_p u-\frac{1}{u^\delta}&\displaystyle=f(x,u)\;&\displaystyle\mbox{in}\;\Omega,\\
\displaystyle u&\displaystyle=0\;&\displaystyle\mbox{on } \partial\Omega.
\end{array}\right.\label{eqStat}
\end{equation}
which are defined by \eqref{subsol} and \eqref{supsol} where $\eta>0$ is small enough and $M>0$ is large enough so that $\underline{u}\leq u_0\le\overline{u}$. Note that it is possible since $u_0\in {\mathcal C}\cap W^{1,p}_0(\Omega)$.
Thus, let $u$ the solutions to $({\rm P}_t)$  and $u_1$, $u_2$ the solutions to $({\rm P}_t)$ with initial data $u_0=\underline{u}$ and $u_0=\overline{u}$ respectively, see Theorem~\ref{3bis}. From \eqref{subsol} and \eqref{supsol}, we have that 
\begin{equation}
\label{DA}
\underline{u},\;\overline{u}\in\overline{{\mathcal D}(A)}^{L^\infty(\Omega)}.
\end{equation}
 Indeed, let $f,g\in W^{-1,\frac{p}{p-1}}(\Omega)$ defined by $f\eqdef A\underline{u}\leq 0$, $g\eqdef A\overline{u}\geq 0$ and $(u_n)_{n\in \N},\,(v_n)_{n\in \N}$ two sequences of ${\mathcal D}(A)$ defined by
\begin{equation}
\nonumber
Au_n=f_n\eqdef\max\{f,-n\},\quad Av_n=g_n\eqdef\min\{g,n\}.
\end{equation}
From the weak comparison principle, we have that $(u_n)_{n\in \N}$ is nonincreasing and $(v_n)_{n\in \N}$ is nondecreasing. Moreover, since $f_n\to f$ and $g_n\to g$ in $W^{-1,\frac{p}{p-1}}(\Omega)$ as $n\to +\infty$, $u_n\to \underline{u}$ and $v_n\to\overline{u}$ in $W^{1,p}_0(\Omega)$ as $n\to +\infty$. Therefore, $u_n\to \underline{u}$ and $v_n\to \overline{u}$ a.e. in $\Omega$ as $n\to\infty$. Consequently, using Dini's Theorem, we get that $u_n\to\underline{u}$ and $v_n\to\overline{u}$ in $L^\infty(\Omega)$ as $n\to +\infty$.
\noindent From \eqref{DA} and Theorem~\ref{3bis}, we obtain that $u_1(t)$ and $u_2(t)\in C([0,T];\,C_0(\overline{\Omega}))$. Furthermore, since $\underline{u}$, $\overline{u}\in{\mathcal C}$ are subsolution and supersolution respectively to \eqref{eqStat}, we have that the sequence $(\underline{u}^n)_{n\in \N}$ (resp. $(\overline{u}^n)_{n\in \N}$) defined in \eqref{iterativeequation} with $u_0=\underline{u}$ ( $u_0=\overline{u}$ resp.,) is nondecreasing (nonincreasing resp.,) for any $0<\Delta_t<1/K$ where $K>0$ is the Lipschitz constant of $f$ on $[\underline{u},\overline{u}]$. Moreover, the sequence $(u^n)_{n\in \N}$ defined by \eqref{iterativeequation} satisfies $\underline{u}^n\leq u^n\leq \overline{u}^n$ and it follows that $u_1(t)\leq u(t)\leq u_2(t)$ and that $t\to u_1(t)$ ($t\to u_2(t)$ resp.,) is nondecreasing (nonincreasing resp.,) and converges a.e. in $\Omega$
  to $u_1^\infty$ ($u_2^\infty$ resp.), as $t\to \infty$. 
From the semigroup theory we have $u_1^\infty=\lim_{t'\to +\infty} S(t'+t)(\underline{u})=S(t)(\displaystyle\lim_{t'\to +\infty}S(t')\underline{u})=S(t)u_1^\infty$ and analogously $u_2^\infty=S(t)u_2^\infty$, where $S(t)$ is the semigroup on $L^\infty(\Omega)$ generated by the evolution equation, and then $u_1^\infty$ and $u_2^\infty$ are stationary solutions to $({\rm P}_t)$. From Theorem~\ref{2bis}, we get that $u_1^\infty=u_2^\infty=u_\infty\in C(\overline{\Omega})$. Therefore, from Dini's Theorem we get that
\begin{equation}
\label{monotonecompactness}
u_1(t)\to u_\infty,\;\; u_2(t)\to u_\infty\;\;\mbox{ in $L^\infty(\Omega)$ as}\; t\to\infty,
\end{equation}
and then \eqref{convasympt} follows since $u_1(t)\leq u(t)\leq u_2(t)$.
\end{proof}
\section{The non degenerate case: $p=2$}\label{pegal2}
We start by proving the first part, as well as points (iv) and (v), of Theorem \ref{4}, namely:
\begin{theorem}
\label{distributionsol}
Let $0<\delta$, $T>0$, $u_0\in{\mathcal C}$ and let $f$ satisfy assumptions in Theorem~\ref{3bis}. Then there exists a unique solution $u\in C([0,T];L^2( \Omega))\cap L^\infty(Q_T)$ such that $u(t)\in \mathcal{C}$ a.e. $t\in [0,T]$, in the sense of distributions to
\begin{eqnarray}
\label{distri-equation-fu}
\left\{\begin{array}{lc}
&\displaystyle u_t-\Delta u-\frac{1}{u^\delta}=f(x,u)\quad\mbox{in}\;Q_T,\\
&\displaystyle u=0\;\mbox{ on } \Sigma_T,\quad u(0)=u_0\;\mbox{in}\;\Omega.
\end{array}\right.
\end{eqnarray}
Moreover, points  (iv) and (v) of Theorem \ref{4} are satisfied.
\end{theorem}
\begin{proof}
First, for $g\in L^\infty(Q_T)$ let us prove the existence of a weak solution $u\in C([0,T];L^2( \Omega))\cap L^\infty(Q_T)$ in the sense of distributions to
\begin{eqnarray}
\label{distri-equation}
\left\{\begin{array}{lc}
&\displaystyle  u_t-\Delta u-\frac{1}{u^\delta}=g\quad\mbox{in}\;Q_T,\\
&\displaystyle u=0\;\mbox{ on } \Sigma_T,\quad u(0)=u_0\;\mbox{in}\;\Omega.
\end{array}\right.
\end{eqnarray}
Since $u_0\in {\mathcal C}$, for any $\epsilon >0$ there exists $\underline{u}_\epsilon\in C_0(\overline{\Omega})\cap H^1_0(\Omega)$ in the form \eqref{epsisub} and \eqref{epsisub2} with $p=2$ if $\delta\geq 1$ and $\underline{u}_\epsilon=\eta\phi_1$ if $\delta<1$ (with $\eta>0$ small enough) such that $\underline{u}_\epsilon\leq u_0$ and
\begin{equation}
\nonumber
-\Delta\underline{u}_\epsilon-\frac{1}{(\underline{u}_\epsilon+\epsilon)^\delta}\leq -\|g\|_{L^\infty(Q_T)}\quad\mbox{in }\;\Omega.
\end{equation}
In addition, there exists $\overline{u}$ in the form given by \eqref{supsol} such that $u_0\leq \overline{u}$ and verifying 
\begin{eqnarray*}
-\Delta \bar{u}-\frac{1}{\bar{u}^\delta}\geq \|g\|_{L^\infty(Q_T)}\quad\mbox{in}\;\Omega.
\end{eqnarray*}
Then following the method and using estimates from Section~\ref{linearcase} (in particular \eqref{esti2} with $p=2$), we can prove the existence and the uniqueness of a positive solution $u_\epsilon\in L^\infty(Q_T)\cap L^2(0,T;\, H^1_0(\Omega))$, satisfying $\underline{u}_\epsilon\leq u_\epsilon\leq\overline{u}$, to
\begin{eqnarray}
\label{epsilon-approche}
({\rm P}_{\epsilon,t})\left\{\begin{array}{lc}
&\displaystyle u_t-\Delta u-\frac{1}{(u+\epsilon)^\delta}=g\quad\mbox{in}\;Q_T,\\
& \displaystyle u=0\;\mbox{ on } \Sigma_T,\;\;u(0)=u_0\;\mbox{in}\;\Omega.
\end{array}\right.
\end{eqnarray}
From the weak comparison principle, we get that  that if $0\leq\tilde{\epsilon}\leq\epsilon$ then $u_{\epsilon}\leq u_{\tilde{\epsilon}}$ and $u_{\tilde{\epsilon}}+\tilde{\epsilon}\leq u_{\epsilon}+\epsilon$. This last inequality is obtained by remarquing that $v_\epsilon\eqdef u_\epsilon+\epsilon$ and $v_{\tilde{\epsilon}}\eqdef u_{\tilde{\epsilon}}+\tilde{\epsilon}$ obeys
\begin{eqnarray*}
\left\{\begin{array}{lc}
&\displaystyle(v_{\tilde{\epsilon}}-v_{\epsilon})_t-\Delta (v_{\tilde{\epsilon}}-v_{\epsilon})-\frac{1}{(v_{\tilde{\epsilon}})^\delta}+\frac{1}{(v_{\epsilon})^\delta}=0\quad\mbox{in}\;Q_T,\\
&\displaystyle v_{\tilde{\epsilon}}-v_\epsilon\leq 0\;\mbox{ on } \Sigma_T,\quad\mbox{and}\;\; v_{\tilde{\epsilon}}(0)-v_\epsilon(0)\leq 0\;\mbox{in}\;\Omega.
\end{array}\right.
\end{eqnarray*}
It follows that $(u_\epsilon)_{\epsilon>0}$ is a Cauchy sequence in $L^\infty(Q_T)$ and there exists $u\in L^\infty(Q_T)$ satisfying $\underline{u}\eqdef\displaystyle\lim_{\epsilon\to 0^+}\underline{u}_\epsilon\leq u\leq \overline{u}$ and such that $u_\epsilon\to u$ in $L^\infty(Q_T)$ as $\epsilon\to 0^+$. Therefore, $u$ is uniformly in ${\mathcal C}$. Then, passing to the limit as $\epsilon\to 0^+$, it is easy to get from \eqref{epsilon-approche} that $u$ is a solution in the sense of distributions to \eqref{distri-equation}. Finally, since $(u_\epsilon)_{\epsilon>0}$ is also a Cauchy sequence in $L^\infty(0,T;\,L^2(\Omega))$, and since one has $u_\epsilon\in C([0,T];\,L^2(\Omega))$ (by regularity results for the heat equation) then $u\in C([0,T];\,L^2(\Omega))$. To prove the uniqueness of $u$, let us suppose that $v\in L^\infty(Q_T)\cap C([0,T],L^2(\Omega))$, $v(t)\in {\mathcal C}$ a.e., is another solution (in the sense of distributions) to \eqref{distri-equation}. Then su
 bstracting the equation satisfied by $v$ to
   \eqref{distri-equation} and doing the dual product in $H^2(\Omega)\cap H^1_0(\Omega)$ with $(-\Delta)^{-1}(u-v)$ yield
\begin{eqnarray*}
\frac{d}{dt}\|u(t)-v(t)\|_{H^{-1}(\Omega)}^2&+&2\|u(t)-v(t)\|_{L^2(\Omega)}^2\\
&+&2\int_\Omega(\frac{1}{v(t)^\delta}-\frac{1}{u(t)^\delta})(-\Delta)^{-1}(u(t)-v(t)){\rm d}x=0,
\end{eqnarray*}
where $\|\cdot\|_{H^{-1}(\Omega)}\eqdef\sqrt{((-\Delta)^{-1}\cdot |\cdot)_{H^1_0(\Omega), H^{-1}(\Omega)}}$ is a norm on $H^{-1}(\Omega)$. Note that from Hardy's inequality the last above term is well-defined since for a.e. $t\in [0,T]$ we have $(-\Delta)^{-1}(u(t)-v(t))\in H^2(\Omega)\cap H_0^1(\Omega)$ ($u(t)$, $v(t)\in \mathcal{C}$ implies $\frac{1}{v(t)^\delta}-\frac{1}{u(t)^\delta}\in H^{-\frac{3}{2}+\eta}(\Omega)\subset (H^2(\Omega)\cap H_0^1(\Omega))'$ for $0<\eta<\frac{2}{1+\delta}$ because $H_0^{\frac{3}{2}-\eta}(\Omega)$ is a subset of $H^2(\Omega)\cap H_0^1(\Omega)$ (see {\sc Grisvard} \cite{Gr} and Lemma \eqref{Hardy} below), and it is positive since $(-\Delta)^{-1}$ is monotone. Then integrating over $(0,T)$ and taking into account $u(0)=v(0)=u_0\in H^{-1}(\Omega)$ yields $u\equiv v$.
\par Let us now prove the existence and uniqueness of the solution to \eqref{distri-equation-fu}. For that we now consider $\underline{u}$, $\overline{u}$ defined by \eqref{subsol}, \eqref{supsol} (subsolution and supersolution to the stationary equation \eqref{distri-equation}, resp.) and obeying $\underline{u}\leq u_0\leq  \overline{u}$. For $z_1$, $z_2$ in $\{\underline{u}\leq u\leq \overline{u}\}\cap C([0,T];\,L^2(\Omega))$ we consider $u_\epsilon^1$ (resp. $u_\epsilon^2$) the solution to $({\rm P}_{\epsilon,t})$ for $g=f(x,z_1)$ ($g=f(x,z_2)$, resp.). By multiplying by $u_\epsilon^1-u_\epsilon^2$ the equation satisfied by $u_\epsilon^1-u_\epsilon^2$, integrating over $(0,t)$, and taking into account that $f(x,\cdot)$ is lipschitz in $[0,\|\overline{u}\|_{\infty}]$, we obtain
$$
\int_\Omega|u_\epsilon^1(t)-u_\epsilon^2(t)|^2{\rm d}x+2\int_0^t\int_{\Omega}|\nabla(u_\epsilon^1-u_\epsilon^2)|^2{\rm d}x{\rm d}t\leq  2\omega\int_0^t\int_{\Omega}|z_1(t)-z_2(t)|^2{\rm d}x{\rm d}t,
$$
for a positive constant $\omega>0$ depending on $\|\overline{u}\|_{\infty}$ but not in $z_1$, $z_2$. From the weak comparison principle we deduce that $\underline{u}_\epsilon \leq u_\epsilon^i \leq  \overline{u}$, $i=1,2$, and
by passing to the limit $\epsilon\to 0^+$ we obtain that the solution to \eqref{distri-equation} for $g=f(x,z_1)$ (resp. $g=f(x,z_2)$) obeys:
\begin{equation}
\int_\Omega|u^1(t)-u^2(t)|^2{\rm d}x\leq 2\omega \int_0^t\int_{\Omega}|z_1(t)-z_2(t)|^2{\rm d}x{\rm d}t\quad \mbox{ and }\quad \underline{u}\leq u^i\leq  \overline{u}, \; i=1,2.\label{estunicite}
\end{equation}
Then by applying the fixed point Theorem in $\{\underline{u}\leq u\leq \overline{u}\}\cap C([0,T];\,L^2(\Omega))$ for $T=T(\|\overline{u}\|_\infty)>0$ small enough, we get the existence of a weak solution (in the sense of distributions), $u\in \{\underline{u}\leq u\leq \overline{u}\} \cap C([0,T];\,L^2(\Omega))$, to \eqref{solutionp=2}. Thus, since $f(x,\cdot)$ is lipschitz in $[0,\|\overline{u}\|_{\infty}]$, we can extend $u$ in $[0,+\infty)$ and $u\in C([0,+\infty);\,L^2(\Omega))$. Note that the uniqueness of the solution is an obvious consequence of \eqref{estunicite}. Finally, if $u_\epsilon^i$ is the solution to \eqref{epsilon-approche} with $u_0=u_{0,i}$ and $g=f(x,u_i)$, $i=1,2$, then by multiplying by $u_\epsilon^1-u_\epsilon^2$ the equation satisfied by $u_\epsilon^1-u_\epsilon^2$ and integrating over $(0,t)$ we obtain:
\begin{eqnarray*}
\int_\Omega|u_\epsilon^1(t)-u_\epsilon^2(t)|^2{\rm d}x+2\lambda_1 \int_0^t\int_\Omega |u_\epsilon^1-u_\epsilon^2|^2{\rm d}x{\rm d}t&\leq& \\
2\int_0^t\int_{\Omega}(f(x,u_1)-f(x,u_2))(u_\epsilon^1-u_\epsilon^2){\rm d}x{\rm d}t&+& \int_\Omega (u_{0,1}-u_{0,2})^2{\rm d}x,
\end{eqnarray*}
where $\lambda_1>0$ is the first eigenvalue of the Dirichlet Laplacian ($\int_\Omega |\nabla z|^2{\rm d}x\geq \lambda_1 \int_\Omega |z|^2{\rm d}x$ for all $z\in H_0^1(\Omega)$). Then \eqref{eqtomega} is obtained  by passing to the limit $\epsilon\to 0^+$ and by applying Gronwall Lemma. 
\end{proof}
\bigskip
\par
Next, we discuss the regularity of solutions to \eqref{distri-equation}. Precisely, we prove the following result:
\begin{theorem}
\label{regularitydistri-equation}
Under the assumptions of Theorem \ref{distributionsol}, $\forall\eta>0$ small enough, we have that any solution $u$ to \eqref{distri-equation} satisfies:
\begin{itemize}
\item[(i)] if $\delta<1$ and $u_0\in{\mathcal C}\cap H^{2-\eta}(\Omega)$, then $u\in C([0,T];\, H^{2-\eta}(\Omega))$;
\item[(ii)] if $\frac{1}{2}\leq \delta<1$ and $u_0\in{\mathcal C}\cap H^{\frac{5}{2}-\delta-\eta}(\Omega)$, then $u\in C([0,T];\, H^{\frac{5}{2}-\delta-\eta}(\Omega))$;
\item[(iii)] if $\delta\geq 1$ and $u_0\in {\mathcal C}\cap H^{\frac{1}{2}+\frac{2}{\delta+1}-\eta}(\Omega)$, then $u\in C([0,T];\, H^{\frac{1}{2}+\frac{2}{\delta+1}-\eta}(\Omega))$. 
\end{itemize}
\end{theorem}
To prove Theorem~\ref{regularitydistri-equation}, we will use the interpolation theory in Sobolev spaces (see {\sc Grisvard} \cite{Gr}, {\sc Triebel} \cite{Tr}) and Hardy inequalities. In this regard, we adopt the following notation: let $X,Y$ be two banach spaces, by $X\subset Y$ we mean that $X$ is continuously imbedded in $Y$. Let us recall some definitions and properties of interpolation spaces theory (see {\sc Triebel} \cite{Tr} for more details):
\begin{definition}
Let $A_0$, $A_1$ be two Banach spaces. For $\theta\in (0,1)$, $1\leq q<\infty$, following the $K$\--method due to {\sc J.~Peetre}, we can define the following interpolation space:
\begin{equation}
\nonumber
\displaystyle (A_0,A_1)_{\theta,q}=\left\{a\in A_0+A_1\big{|}\,\|a\|_{(A_0,A_1)_{\theta,q}}\eqdef\left(\int_0^\infty t^{-\theta}K(t,a)^q{\rm d}t\right)^{\frac{1}{q}}<\infty\right\},
\end{equation}
and for $q=\infty$,
\begin{equation}
\nonumber
\displaystyle (A_0,A_1)_{\theta,\infty}=\left\{a\in A_0+A_1\big{|}\,\|a\|_{(A_0,A_1)_{\theta,\infty}}\eqdef \displaystyle\sup_{0<t<\infty}t^{-\theta}K(t,a)^q<\infty\right\},
\end{equation}
with
\begin{equation}
\nonumber
K(t,a)\eqdef\displaystyle\inf_{a=a_1+a_2}\left(\|a_0\|_{A_0}+t\|a_1\|_{A_1}\right)\quad\mbox{for}\; a\in A_0+A_1.
\end{equation}
\end{definition}
\begin{remark}
Classical examples of interpolation spaces are the Lorentz spaces $L^{p,q}(\Omega)=(L^1(\Omega),L^\infty(\Omega))_{\frac{p-1}{p},q}$ and $L^{p,\infty}(\Omega)$ known as the weak $L^p$\--space of Marcienkewicz (note that $L^{p,p}(\Omega)=L^p(\Omega)$).
\end{remark}
We now recall some properties satisfied by interpolation spaces that we will use in the proof of Theorem~\ref{regularitydistri-equation} (see \cite[Thm. 1.3.3.1, Par. 1.3.3, p.25 and Par. 1.12.2, Thm 1.12.2.1, p.69]{Tr}):
\begin{proposition}
\label{interpolationp}
Let $A_0$, $A_1$ be two Banach spaces. Let $\theta\in (0,1)$, $1\leq q<\infty$.
Then,
\begin{itemize}
\item[(i)] $A_0\cap A_1\subset (A_0,A_1)_{\theta,q}\subset A_0+A_1$ and $(A_0,A_1)_{\theta,q}=(A_1,A_0)_{1-\theta,q}$;
\item[(ii)] if $q\leq \tilde{q}<\infty$, then 
\begin{equation}
\nonumber
(A_0,A_1)_{\theta,1}\subset(A_0,A_1)_{\theta,q}\subset (A_0,A_1)_{\theta,\tilde{q}}\subset(A_0,A_1)_{\theta,\infty};
\end{equation}
\item[(iii)] if $A_0\subset A_1$, then for $\theta <\tilde{\theta}<1$, $1\leq \tilde{q}\leq \infty$, $(A_0,A_1)_{\theta,q}\subset (A_0,A_1)_{\tilde{\theta},\tilde{q}}$ holds;
\item[(iv)] $\exists c_{\theta,q}$ such that $\forall a\in A_0\cap A_1$, $\|a\|_{(A_0,A_1)_{\theta,q}}\leq c_{\theta,q}\|a\|_{A_0}^{1-\theta}\|a\|_{A_1}^\theta$.
\item[(v)] (Duality Theorem) If $A_0\cap A_1$ is dense both in $A_0$ and $A_1$ then $$(A_0,A_1)_{\theta,p}'=(A_0',A_1')_{\theta,\frac{p}{p-1}}.$$
\item[(vi)] (Interpolation Theorem) Let $B_0$, $B_1$ be two Banach spaces. If $L$ is a linear continuous mapping from $A_0$ to $B_0$ as well as from $A_1$ to $B_1$ then $L$ is linear continuous from $(A_0,A_1)_{\theta,q}$ to $(B_0,B_1)_{\theta,q}$.
\end{itemize}
\end{proposition}
Finally, we give the definition and some basic properties of the space of traces introduced by {\sc J.~L. Lions} (see {\sc triebel} \cite{Tr} and {\sc Benssoussan-Da Prato-Delfour-Mitter} \cite{BePrDeMi} for more general details):
\begin{definition}
\label{deftracespace}
Let $X_0$, $X_1$ two Banach spaces such that $X_0\subset X_1$ (densely). For $1<q<\infty$, let $W_q(0,\infty,X_0,X_1)$ and $T(q,X_0,X_1)$ the Banach spaces
$$
W_q(0,\infty,X_0,X_1) \eqdef  \left\{u\in L^q(0,\infty;\,X_0)|\,u_t\in L^q(0,\infty;\,X_1 )\right\},
$$
equipped with the norm $\left(\int_0^\infty\|u\|_{X_0}^q+\|u_t\|_{X_1}^q{\rm d}t\right)^{\frac{1}{q}}$ and
\begin{equation}
\nonumber
T(q,X_0,X_1)\eqdef\left\{x\in X_1\,:\,\exists u\in W_q(0,\infty,X_0,X_1),\, u(0)=x\right\},
\end{equation}
endowed by the norm
\begin{equation}
\nonumber
\displaystyle\|x\|_{T(q,X_0,X_1)}\eqdef \inf\left\{\|u\|_{W_q(0,\infty,X_0,X_1)}\,:\,u(0)=x\right\}.
\end{equation}
\end{definition}
The space of traces $T(q,X_0,X_1)$ (under the assumptions given in Definition~\ref{deftracespace}) has the following properties (see \cite[Thm. 1.8.2.1 (2), Par.1.8.2, p.44]{Tr} or \cite[Part II, Chap. 1, Par. 4.2]{BePrDeMi}):
\begin{proposition}
\label{propertiestracespace}
Assume that the conditions given in Definition~\ref{deftracespace} hold. Then, we have:
\begin{enumerate}
\item $X_0\subset T(q,X_0,X_1)\subset X_1$;
\item $u\in C\left([0,\infty);T(q,X_0,X_1)\right)$ for any $u\in W_q(0,\infty, X_0, X_1)$;
\item ({\sl Equivalence Theorem}) let $0<\theta=\frac{1}{q}<1$, then $T(q,X_0,X_1)\eqsim (X_0,X_1)_{\theta,q}$.
\end{enumerate}
\end{proposition}
We recall now some basic facts about fractional powers of $-\Delta$ with domain $\mathcal{D}(-\Delta)=H^2(\Omega)\cap H_0^1(\Omega)$, i.e. $(-\Delta)^\theta$ with domain  ${\mathcal D}((-\Delta)^\theta)$ in $L^2(\Omega)$  for $\theta\in [0,1]$ (see for instance \cite[Par. 1.15.1, p. 98 and Par. 1.18.10, p. 141]{Tr} or \cite[Chap. 2, Par. 2.6]{Pazy}  for the definition):
\begin{proposition}
\label{fractionalpower}
Let $\theta\in [0,1]$.
\begin{enumerate}
\item ${\mathcal D}((-\Delta)^\theta)=({\mathcal D}(\Delta),L^2(\Omega))_{1-\theta,2}$
\item $\displaystyle{\mathcal D}((-\Delta)^\theta)=\left\{\begin{array}{ll}
 H^{2\theta}(\Omega)\;\;&\mbox{if}\; 0\leq\theta<\frac{1}{4},\\
\tilde{H}^{\frac{1}{2}}(\Omega)\;\;&\mbox{if }\theta=\frac{1}{4},\\
H^{2\theta}_0(\Omega)\;\;&\mbox{if}\; \frac{1}{4}<\theta\leq 1.
\end{array}\right.$\\
where $\tilde{H}^{\frac{1}{2}}(\Omega)\eqdef\{v\in H^{\frac{1}{2}}(\Omega)|\,\displaystyle d(x)^{-\frac{1}{2}} v \in L^2(\Omega)\}$.
\item $-\Delta^\theta$ is an isomorphism from ${\mathcal D}(-\Delta)$ onto ${\mathcal D}((-\Delta)^{1-\theta})$ as well as from $L^2(\Omega)$ onto $({\mathcal D}((-\Delta)^\theta))'$ (the dual space of  $({\mathcal D}((-\Delta)^\theta))$).
\end{enumerate}
\end{proposition}
\begin{remark}
Note that since we are in a Hilbertian framework, the real interpolation space $({\mathcal D}(\Delta),L^2(\Omega))_{1-\theta,2}$ coincides with the complex interpolation space $[{\mathcal D}(\Delta),L^2(\Omega)]_{1-\theta}$, see \cite[Par. 1.9, p. 55 and Par. 1.8.10, Rem. 3 p. 143]{Tr}.
\end{remark}
\begin{proof}
The first point follows from the fact that $-\Delta$ is positive and sel-adjoint, see \cite[Thm. 1.18.10.1, par. 1.18.1 p.141]{Tr}. The second point is a consequence of the characterization of $(H^2(\Omega)\cap H_0^1(\Omega),L^2(\Omega))_{1-\theta,2}$ obtained from 
\cite[Thm. 4.3.3.1, Par. 4.3.3, p.321]{Tr} or 
from \cite{GRISVARD}. The first part of the last point follows from \cite[Thm. 1.15.2.1, p.101]{Tr} and the second part is deduced with a duality argument (combined with the self-adjointness of $-\Delta$) from the fact that $-\Delta$ is an isomorphism from ${\mathcal D}((-\Delta)^{\theta})$ onto $L^2(\Omega)$.
\end{proof}
For $1<q<\infty$, setting the Banach space
$$
{\mathcal X}_{q,\theta,T}\eqdef W_q(0,T;\,{\mathcal D}((-\Delta)^{1-\theta}),\,({\mathcal D}((-\Delta)^\theta))'),
$$
we have the following result:
\begin{lemma}
\label{maximalregularity}
Let $\theta\in [0,1)$ and $q>\frac{2}{1-\theta}$. let $L_T$ the linear operator defined by $L_T(f)\eqdef z$, where $z$ is the solution to
\begin{eqnarray*}
\left\{\begin{array}{c}
\displaystyle z_t-\Delta z=f\quad\mbox{in}\; Q_T,\\
\displaystyle z=0\;\mbox{on }\Sigma_T;\quad z(0)=0\;\mbox{in}\; \Omega.
\end{array}\right.
\end{eqnarray*} 
Then $L_T$ is a bounded operator from $L^q(0,T;\,({\mathcal D}((-\Delta)^\theta)')$ into ${\mathcal X}_{q,\theta,T}$ as well as from $L^q(0,T;\,({\mathcal D}((-\Delta)^\theta)')$ into $C([0,T],\,{\mathcal D}((-\Delta)^{1-\theta-\frac{2}{q}}))$.
\end{lemma}
\begin{proof}
Since $(-\Delta)$ generates an analytic semigroup, denoted by $e^{t\Delta}$, in $L^2(\Omega)$, then $L_T$ is a continuous operator from $L^q(0,T;\, L^2(\Omega))$ onto $W_q(0,T;\,{\mathcal D}(-\Delta),$ $L^2(\Omega))$ (see \cite[Par. 4.2]{DeHiPr2003}). Therefore, since $(-\Delta)^\theta$ is an isomorphism from ${\mathcal D}(-\Delta)$ onto ${\mathcal D}((-\Delta)^{1-\theta})$ and from $L^2(\Omega)$ onto ${\mathcal D}((-\Delta)^\theta)'$ and from the fact that $(-\Delta)^\theta$ and $L$ commute, the first statement of the Lemma follows. We prove now the second statement. Let $T'>T$ and set
\begin{eqnarray*}
\hat{f}\in L^\infty(Q_{T'})\;\mbox{ defined by}\;\hat{f}(t,\cdot)\left\{\begin{array}{lc}
&f(t,\cdot)\;\mbox{if }t\leq T\\
&0\;\mbox{ otherwise}
\end{array}\right.
\end{eqnarray*}
and $\hat{u}\eqdef L_{T'}(\hat{f})$. Then, we define the extension function $u_1(t,\cdot)= \hat{u}(t,\cdot)$  if $t\leq T'$ and $u_1(t,\cdot)=0$ otherwise. We have that $u_1|_{Q_T}=L_T(f)=u$. Let
$\chi\,:\,\R^+\to \R^+$ be a smooth non\--increasing function such that $\chi\equiv 1$ on $[0,T]$ and $\chi\equiv 0$ on $[T',+\infty)$. Then, for $u\in {\mathcal X}_{q,\theta,T}$ we have $\chi u_1\in {\mathcal X}_{q,\theta,\infty}$. Thus, from Proposition~\ref{propertiestracespace}, $\chi u_1\in C([0,+\infty),\left({\mathcal D}((-\Delta)^{1-\theta}),({\mathcal D}((-\Delta)^\theta))'\right)_{\frac{1}{q},q})$ and therefore, 
$$
u\in C([0,T],\left({\mathcal D}((-\Delta)^{1-\theta}),({\mathcal D}((-\Delta)^\theta))'\right)_{\frac{1}{q},q}).
$$
Furthermore, from the interpolation Theorem we have
\begin{equation}
\nonumber
\left({\mathcal D}((-\Delta)^{1-\theta}),({\mathcal D}((-\Delta)^\theta))'\right)_{\frac{1}{q},q}=(-\Delta)^\theta\left(({\mathcal D}(-\Delta),L^2(\Omega))_{\frac{1}{q},q}\right)
\end{equation}
and from Propositions~\ref{interpolationp} and \ref{fractionalpower}, we get
\begin{equation}
\nonumber
\left({\mathcal D}(-\Delta),L^2(\Omega)\right)_{\frac{1}{q},q}\subset\left({\mathcal D}(-\Delta),L^2(\Omega)\right)_{\frac{2}{q},1}\subset\left({\mathcal D}(-\Delta),L^2(\Omega)\right)_{\frac{2}{q},2}={\mathcal D}((-\Delta)^{1-\frac{2}{q}}).
\end{equation}
From above, it follows that 
$$
u\in C([0,T],{\mathcal D}((-\Delta)^{1-\theta-\frac{2}{q}})).
$$
\end{proof}
Let us also recall the following Hardy type inequalities which can be obtained from \cite[Par. 3.2.6, Lem. 3.2.6.1,  p.259]{Tr}.
\begin{lemma}\label{Hardy} Let $s\in [0,2]$ such that $s\neq \frac{1}{2}$ and $s\neq \frac{3}{2}$. Then the following generalisation of Hardy's inequality holds:
\begin{equation}
\|d^{-s} g\|_{L^{2}(\Omega)}\leq C \|g\|_{H^{s}(\Omega)}\quad \mbox{ for all }g\in H_0^s(\Omega). \label{eq4b}
\end{equation}
\end{lemma}
\bigskip
\par\noindent We now prove Theorem~\ref{regularitydistri-equation}.
\begin{proof}
First, notice that with the notation of  Lemma~\ref{maximalregularity} the solution $u$ of \eqref{distri-equation-fu} obeys 
$$u(t)=e^{\Delta t}u_0+L_T\left(\frac{1}{u^\delta}+f(x,u)\right)(t).$$
In the following we suppose that $0<\eta<\frac{1}{1+\delta}$. If $0\leq \delta<\frac{1}{2}$ then since $u\in {\mathcal C}$ we have $\frac{1}{u^\delta}=O(\frac{1}{d(x)^\delta})$ and it implies that $\frac{1}{u^\delta}+f(x,u)\in L ^2(\Omega)$. Then using Lemma~\ref{maximalregularity} we obtain $L_T(\frac{1}{u^\delta}+f(x,u))\in C([0,T],H^{2-\eta}(\Omega))$.
If $\frac{1}{2}\leq\delta<1$,  then from $u\in {\mathcal C}$ we have $\frac{1}{u^\delta}=O(\frac{1}{d(x)^\delta})$ and by Lemma \ref{Hardy} it implies that $\frac{1}{u^\delta}+f(x,u)\in C([0,T],(H_0^{\delta-\frac{1}{2}+\frac{\eta}{2}}(\Omega))')= C([0,T],({\mathcal D}((-\Delta)^{\frac{\delta}{2}-\frac{1}{4}+\frac{\eta}{4}}))')$ (since $\frac{\delta}{2}-\frac{1}{4}+\frac{\eta}{4} \in (0,\frac{3}{4})$). Then Lemma~\ref{maximalregularity} with $q=\frac{4}{\eta}$ gives $L_T(\frac{1}{u^\delta}+f(x,u))\in C([0,T],H^{\frac{5}{2}-\delta-\eta}(\Omega))$. 
Next, if $\delta\geq 1$ then by Lemma \ref{Hardy} we have $\frac{1}{u^\delta}+f(x,u)\in C([0,T],(H_0^{\frac{2\delta}{\delta+1}-\frac{1}{2}+\frac{\eta}{2}}(\Omega))')= C([0,T],({\mathcal D}((-\Delta)^{\frac{\delta}{\delta+1}-\frac{1}{4}+\frac{\eta}{4}}))')$ (since $\frac{\delta}{\delta+1}-\frac{1}{4}+\frac{\eta}{4}\in (0,\frac{3}{4})$). Therefore, using Lemma~\ref{maximalregularity} with $q=\frac{2}{\eta}$ we get that $L_T(\frac{1}{u^\delta}+f(x,u))\in C([0,T],H^{\frac{5}{2}-\frac{2\delta}{\delta+1}-\eta}(\Omega))$. As a consequence, if $u_0\in X_\eta$ defined by 
\begin{equation}
X_\eta=\displaystyle\left\{\begin{array}{ll}
\displaystyle H^{2-\eta}(\Omega)\;&\mbox{if}\; \delta<\frac{1}{2},\\
\displaystyle  H^{\frac{5}{2}-\delta-\eta}(\Omega)\;&\mbox{if}\; \frac{1}{2}\leq\delta<1,\\
\displaystyle H^{\frac{5}{2}-\frac{2\delta}{\delta+1}-\eta}(\Omega)\;&\mbox{if}\; 1\leq\delta,
\end{array}\right.\label{eqXeta}
\end{equation}
then $t\to u(t)=e^{t\Delta}u_0+L_T(\frac{1}{u^\delta}+f(x,u))(t)\in C([0,T],X_\eta)$.
\end{proof}
\bigskip
\par\noindent
Then Theorem \ref{4} follows from Theorems \ref{distributionsol} and \ref{regularitydistri-equation}. We end this section by proving Theorem~\ref{5}.
\bigskip
\begin{proof}
Let $\delta<3$ and $u_0\in H^1_0(\Omega)\cap{\mathcal C}$. Then, from Theorem~\ref{3bis}, Theorem~\ref{3ter} and Theorem~\ref{4}, the solution to \eqref{solutionp=2} is unique, belongs to $C([0,+\infty),\, H^1_0(\Omega))$ and satisfies $\underline{u}\leq u\leq\overline{u}$ (with $\underline{u}$, $\overline{u}$ as in the proof of Theorem~\ref{4}), and  
\begin{equation}
\nonumber
u(t)\to u_\infty\quad\mbox{in}\;L^\infty(\Omega)\;\;\mbox{as}\; t\to +\infty,
\end{equation}
where $u_\infty$ is given by Theorem~\ref{2bis} with $p=2$. To complete the proof of Theorem~\ref{5}, let us show that
\begin{equation}
\label{asymptoticsobolev}
u(t)\to u_\infty\quad\mbox{in}\;H^1_0(\Omega)\;\;\mbox{as}\; t\to +\infty.
\end{equation}
For that, let $0<\eta<\min(\frac{1}{2},\frac{3-\delta}{2(1+\delta)})$ and note that $X_\eta$ defined by \eqref{eqXeta} is compactely embedded in $H^1_0(\Omega)$. Moreover, for any $t>0$, using Lemma~\ref{maximalregularity} as in the proof of Theorem~\ref{regularitydistri-equation} and the fact that $u_0\in {\mathcal C}\cap H^1_0(\Omega)$, we have for $t'\geq t$,
\begin{equation}
\label{compactnesstraj}
t'\to S(t')u_0=e^{t'\Delta}u_0+L_{t'}(\frac{1}{u^\delta}+f(x,u))\in C([t,+\infty),\,X_\eta),
\end{equation}
and since $\underline{u}\leq u\leq\overline{u}$, 
\begin{equation}
\label{compactnesstraj2}
\sup_{[t,+\infty]}\|u(t)\|_{X_\eta}<+\infty.
\end{equation}
From the compactness embedding of $X_\eta$ in $H^1_0(\Omega)$ together with \eqref{compactnesstraj2}, \eqref{asymptoticsobolev} follows.
\end{proof}
\bibliography{ma_biblio}

\begin{thebibliography}{10}

\bibitem{Anane-1}
Aomar Anane.
\newblock Simplicit\'e et isolation de la premi\`ere valeur propre du
  {$p$}-laplacien avec poids.
\newblock {\em C. R. Acad. Sci. Paris S\'er. I Math.}, 305(16):725--728, 1987.

\bibitem{Anane-2}
Aomar Anane.
\newblock Etude des valeurs propres et de la r\'esonance pour l'op\'erateur
  $p$-laplacien.
\newblock {\em Th\`ese de doctorat, Universit\'e Libre de Bruxelles}, 1988.

\bibitem{ArGo}
Carlos Aranda and Tomas Godoy.
\newblock Existence and multiplicity of positive solutions for a singular
  problem associated to the {$p$}-{L}aplacian operator.
\newblock {\em Electron. J. Differential Equations}, pages No. 132, 15 pp.
  (electronic), 2004.

\bibitem{Ba}
Viorel Barbu.
\newblock {\em Nonlinear differential equations of monotone types in {B}anach
  spaces}.
\newblock Springer Monographs in Mathematics. Springer, New York, 2010.

\bibitem{BePrDeMi}
A.~Bensoussan, G.~Da~Prato, M.~C. Delfour, and S.~K. Mitter.
\newblock {\em Representation and control of infinite dimensional systems}.
\newblock Systems \& Control: Foundations \& Applications\. Birkh\"auser Boston
  Inc., Boston, MA, second edition, 2007.

\bibitem{CrRaTa}
M.~G. Crandall, P.~H. Rabinowitz, and L.~Tartar.
\newblock On a {D}irichlet problem with a singular nonlinearity.
\newblock {\em Comm. Partial Differential Equations}, 2(2):193--222, 1977.

\bibitem{CuTa}
Mabel Cuesta and Peter Tak{\'a}{\v{c}}.
\newblock A strong comparison principle for positive solutions of degenerate
  elliptic equations.
\newblock {\em Differential Integral Equations}, 13(4-6):721--746, 2000.

\bibitem{DaMo}
Juan D{\'a}vila and Marcelo Montenegro.
\newblock Existence and asymptotic behavior for a singular parabolic equation.
\newblock {\em Trans. Amer. Math. Soc.}, 357(5):1801--1828 (electronic), 2005.

\bibitem{De}
Klaus Deimling.
\newblock {\em Nonlinear functional analysis}.
\newblock Springer-Verlag, Berlin, 1985.

\bibitem{DiSa}
Jes{\'u}s~Ildefonso D{\'{\i}}az and Jos{\'e}~Evaristo Sa{\'a}.
\newblock Existence et unicit\'e de solutions positives pour certaines
  \'equations elliptiques quasilin\'eaires.
\newblock {\em C. R. Acad. Sci. Paris S\'er. I Math.}, 305(12):521--524, 1987.

\bibitem{Be}
E.~DiBenedetto.
\newblock {$C^{1+\alpha }$} local regularity of weak solutions of degenerate
  elliptic equations.
\newblock {\em Nonlinear Anal.}, 7(8):827--850, 1983.

\bibitem{DrHe}
Pavel Dr{\'a}bek and Jes{\'u}s Hern{\'a}ndez.
\newblock Existence and uniqueness of positive solutions for some quasilinear
  elliptic problems.
\newblock {\em Nonlinear Anal.}, 44(2, Ser. A: Theory Methods):189--204, 2001.

\bibitem{FlTa}
Jacqueline Fleckinger-Pell{\'e} and Peter Tak{\'a}{\v{c}}.
\newblock Uniqueness of positive solutions for nonlinear cooperative systems
  with the {$p$}-{L}aplacian.
\newblock {\em Indiana Univ. Math. J.}, 43(4):1227--1253, 1994.

\bibitem{GhRa}
Marius Ghergu and Vicen{\c{t}}iu~D. R{\u{a}}dulescu.
\newblock {\em Singular elliptic problems: bifurcation and asymptotic
  analysis}, volume~37 of {\em Oxford Lecture Series in Mathematics and its
  Applications}.
\newblock The Clarendon Press Oxford University Press, Oxford, 2008.

\bibitem{GiScTa}
Jacques Giacomoni, Ian Schindler, and Peter Tak{\'a}{\v{c}}.
\newblock Sobolev versus {H}\"older local minimizers and existence of multiple
  solutions for a singular quasilinear equation.
\newblock {\em Ann. Sc. Norm. Super. Pisa Cl. Sci. (5)}, 6(1):117--158, 2007.

\bibitem{GRISVARD}
P.~Grisvard.
\newblock Caract\'erisation de quelques espaces d'interpolation.
\newblock {\em Arch. Rational Mech. Anal.}, 25:40--63, 1967.

\bibitem{Gr}
P.~Grisvard.
\newblock {\em Elliptic problems in nonsmooth domains}, volume~24 of {\em
  Monographs and Studies in Mathematics}.
\newblock Pitman (Advanced Publishing Program), Boston, MA, 1985.

\bibitem{HeMa}
J.~Hern\'andez and F.~J. Mancebo.
\newblock Singular elliptic and parabolic equations.
\newblock {\em Handbook of Differential Equations}, 3:317--400, 2006.

\bibitem{HeMaVe}
Jes{\'u}s Hern{\'a}ndez, Francisco~J. Mancebo, and Jos{\'e}~M. Vega.
\newblock On the linearization of some singular, nonlinear elliptic problems
  and applications.
\newblock {\em Ann. Inst. H. Poincar\'e Anal. Non Lin\'eaire}, 19(6):777--813,
  2002.

\bibitem{Li}
Gary Lieberman.
\newblock Boundary regularity for solutions of degenerate elliptic equations.
\newblock {\em Nonlinear Anal.}, 12(11):1203--1219, 1988.

\bibitem{Linq}
Peter Lindqvist.
\newblock On the equation {${\rm div}\,(\vert \nabla u\vert ^{p-2}\nabla
  u)+\lambda\vert u\vert ^{p-2}u=0$}.
\newblock {\em Proc. Amer. Math. Soc.}, 109(1):157--164, 1990.

\bibitem{Pazy}
A.~Pazy.
\newblock {\em Semigroups of linear operators and applications to partial
  differential equations}, volume~44 of {\em Applied Mathematical Sciences}.
\newblock Springer-Verlag, New York, 1983.

\bibitem{PeSi}
Kanishka Perera and Elves A.~B. Silva.
\newblock On singular {$p$}-{L}aplacian problems.
\newblock {\em Differential Integral Equations}, 20(1):105--120, 2007.

\bibitem{Se}
James Serrin.
\newblock Local behavior of solutions of quasi-linear equations.
\newblock {\em Acta Math.}, 111:247--302, 1964.

\bibitem{Ta}
Peter Tak{\'a}{\v{c}}.
\newblock Stabilization of positive solutions for analytic gradient-like
  systems.
\newblock {\em Discrete Contin. Dynam. Systems}, 6(4):947--973, 2000.

\bibitem{To1}
Peter Tolksdorf.
\newblock On the {D}irichlet problem for quasilinear equations in domains with
  conical boundary points.
\newblock {\em Comm. Partial Differential Equations}, 8(7):773--817, 1983.

\bibitem{To}
Peter Tolksdorf.
\newblock Regularity for a more general class of quasilinear elliptic
  equations.
\newblock {\em J. Differential Equations}, 51(1):126--150, 1984.

\bibitem{Tr}
H.~Triebel.
\newblock {\em Interpolation theory, function spaces, differential operators}.
\newblock Johann Ambrosius Barth, Heidelberg, second edition, 1995.

\bibitem{Va}
J.~L. V{\'a}zquez.
\newblock A strong maximum principle for some quasilinear elliptic equations.
\newblock {\em Appl. Math. Optim.}, 12(3):191--202, 1984.

\bibitem{Wi}
Michael Winkler.
\newblock Nonuniqueness in the quenching problem.
\newblock {\em Math. Ann.}, 339(3):559--597, 2007.

\end{thebibliography}
\bibliographystyle{plain}
\end{document}